\documentclass[reqno,11pt]{amsart}

\usepackage{enumerate}
\usepackage{hyperref}
\usepackage[all]{xy}

\newcommand\al{\alpha}

\newcommand\ze{\zeta}
\newcommand\et{\eta}
\newcommand\io{\iota}

\newcommand\ph{\varphi}

\newcommand\om{\omega}
\newcommand\Ga{\Gamma}

\newcommand\Ph{\Phi}

\newcommand\Om{\Omega}

\newcommand\oo{{\infty}}
\newcommand\x{\times}

\newcommand\N{\mathcal N}
\newcommand\T{\mathcal T}

\newcommand\g{\mathfrak g}
\newcommand\hamoe{\mathfrak{ham}}

\newcommand\ori{\textnormal{or}}
\newcommand\KKS{\textnormal{KKS}}

\newcommand\symp{\textnormal{symp}}

\DeclareMathOperator{\Ad}{Ad}
\DeclareMathOperator{\ad}{ad}
\DeclareMathOperator{\Emb}{Emb}
\DeclareMathOperator{\ev}{ev}
\DeclareMathOperator{\pr}{pr}
\DeclareMathOperator{\Ham}{Ham}
\DeclareMathOperator{\Flag}{Flag}
\DeclareMathOperator{\pFlag}{Fr}
\DeclareMathOperator{\Diff}{Diff}
\DeclareMathOperator{\Symp}{Symp}
\DeclareMathOperator{\id}{id}
\DeclareMathOperator{\Gr}{Gr}

\theoremstyle{plain}
\newtheorem{theorem}{Theorem}[section]
\newtheorem{lemma}[theorem]{Lemma}
\newtheorem{proposition}[theorem]{Proposition}

\theoremstyle{remark}
\newtheorem{remark}[theorem]{Remark}
\newtheorem{example}[theorem]{Example}
\theoremstyle{definition}

\begin{document}

\title{Nonlinear flag manifolds as coadjoint orbits}

\author{Stefan Haller}

\address{Stefan Haller,
         Department of Mathematics,
         University of Vienna,
         Oskar-Mor\-gen\-stern-Platz~1,
         1090~Vienna,
         Austria.}

\email{stefan.haller@univie.ac.at}

\author{Cornelia Vizman}

\address{Cornelia Vizman,
         Department of Mathematics,
         West University of Timi\c soara,
         Bd. V.P\^arvan~4,
         300223-Timi\c soara,
         Romania.}

\email{cornelia.vizman@e-uvt.ro}

\begin{abstract}
A nonlinear flag is a finite sequence of nested closed submanifolds.
We study the geometry of Fr\'echet manifolds of nonlinear flags, in this way generalizing the nonlinear Grassmannians.
As an application we describe a class of coadjoint orbits of the group of Hamiltonian diffeomorphisms that consist of nested symplectic submanifolds, i.e., symplectic nonlinear flags.
\end{abstract}



\maketitle


\section{Introduction}

Let $M$ be a smooth manifold and suppose $S_1,\dotsc,S_r$ are closed smooth manifolds.
A nonlinear flag of type $\mathcal S=(S_1,\dotsc,S_r)$ in $M$ is sequence of nested embedded submanifolds $N_1\subseteq\cdots\subseteq N_r\subseteq M$ such that $N_i$ is diffeomorphic to $S_i$ for all $i=1,\dotsc,r$.
The space of all nonlinear flags of type $\mathcal S$ in $M$ can be equipped with the structure of a Fr\'echet manifold in a natural way and will be denoted by $\Flag_{\mathcal S}(M)$.
The aim of this paper is to study the geometry of this space.

Nonlinear flag manifolds provide a natural generalization of nonlinear Grassmannians which correspond to the case $r=1$.
Nonlinear Grassmannians (a.k.a.\ differentiable Chow manifolds) play an important role in computer vision \cite{BBM14,M16} and continuum mechanics \cite{M20}.
They have also been used to describe coadjoint orbits of diffeomorphism groups.
Nonlinear Grassmannians of symplectic submanifolds have been identified with coadjoint orbits of the Hamiltonian group in \cite{HV}.
Codimension two Grassmannians have been used to describe coadjoint orbits of the group of volume preserving diffeomorphisms \cite{I96,HV}.
Let us also point out that every closed $k$-fold vector cross product on a Riemannian manifold induces an almost K\"ahler structure on the nonlinear Grassmannians of $(k-1)$-dimensional submanifolds \cite{LL07}.

In some applications decorated nonlinear Grassmannians have been considered, that is, spaces of submanifolds equipped with additional data supported on the submanifold.
Functional shapes (fshapes), for instance, may be described as signal functions supported on shapes \cite{CCT,CCT2,CT}.
Weighted nonlinear Grassmannians of isotropic submanifolds have been used to describe coadjoint orbits of the Hamiltonian group \cite{GBV17,Lee,W90}.
Recently, weighted nonlinear Grassmannians of isotropic submanifolds have been identified with coadjoint orbits of the contact group \cite{HV19}.
Decorated codimension one Grassmannians may be used to describe coadjoint orbits of the group of volume preserving diffeomorphisms \cite{GBV19}.
The nonlinear flag manifolds considered in this paper may be regarded as yet another class of decorated Grassmannians.

Some nonlinear flag manifolds have already appeared in the literature too.
Landmark-constrained planar curves, for instance, have been used in a statistical elastic shape analysis framework in \cite{SKBM}.
Landmark-constrained surfaces in the context of shape analysis are being discussed in \cite[Chapter~6]{JKLS}.
An attempt to use the nonlinear flag manifold of surfaces in $\mathbb R^3$ decorated with curves as shape space can be found in \cite{TV}.
Manifolds of weighted nonlinear flags are the object of study in \cite{HV20}.
We hope that the foundational material on nonlinear flag manifolds provided in this paper will prove helpful in future research.

As a first application, we will use nonlinear flag manifolds to describe certain coadjoint orbits of the Hamiltonian group.
To be more explicit, suppose $M$ is a closed symplectic manifold and let $\Flag_{\mathcal S}^\symp(M)$ denote the open subset in $\Flag_{\mathcal S}(M)$ consisting of all symplectic flags of type $\mathcal S$.
The symplectic form on $M$ induces by transgression a symplectic form on the manifold of symplectic nonlinear flags.
The Hamiltonian group $\Ham(M)$ acts on $\Flag_{\mathcal S}^\symp(M)$ in a Hamiltonian fashion with equivariant moment map
\[
J\colon\Flag_{\mathcal S}^\symp(M)\to\hamoe(M)^*.
\]
This moment map is injective and identifies each connected component of $\Flag_{\mathcal S}^\symp(M)$ with a coadjoint orbit of $\Ham(M)$, see Theorem~\ref{teo2} below.

The remaining part of this paper is organized as follows.
Section~\ref{S:Flag} contains a rigorous study of the Fr\'echet manifold $\Flag_{\mathcal S}(M)$ and related principal bundles.
In Section~\ref{S:oFlag} we discuss the oriented analogue, that is, the Fr\'echet manifold of all oriented nonlinear flags, a finite covering of $\Flag_{\mathcal S}(M)$.
In Section~\ref{S:Flagsymp} we study the action of the Hamiltonian group on the open subset of symplectic flags and provide a proof of Theorem~\ref{teo2} mentioned before.

\subsection*{Acknowledgments}
The first author would like to thank the West University of Ti\-mi\-\c soa\-ra for the warm hospitality.
He gratefully acknowledges the support of the Austrian Science Fund (FWF): project numbers P31663-N35 and Y963-N35.
The second author would like to thank the University of Vienna for the warm hospitality.
She was supported by the grant PN-III-P4-ID-PCE-2016-0778 of the Romanian
Ministry of Research and Innovation CNCS-UEFISCDI, within PNCDI III.
Both authors would like to express their gratitude to Martin Bauer and Peter Michor for helpful bibliographical comments.

\section{Manifolds of nonlinear flags}\label{S:Flag}

Let $M$ be a smooth manifold and suppose $S_1,\dotsc,S_r$ are closed smooth manifolds.
In this section we study the space $\Flag_{\mathcal S}(M)$ of all nonlinear flags of type $\mathcal S=(S_1,\dotsc,S_r)$ in $M$, i.e., the space of all nested submanifolds $N_1\subseteq\dots\subseteq N_r$ of $M$ that arise from embedding $S_1,\dotsc,S_r$ into $M$.
We will equip this space with the structure of a Fr\'echet manifold, describe its smooth structure in several ways, and discuss related (principal) bundles systematically.

In Proposition~\ref{L1} we will show that $\Flag_{\mathcal S}(M)$ may be considered as smooth submanifold in the product of nonlinear Grassmannians, $\Gr_{S_1}(M)\times\cdots\times\Gr_{S_r}(M)$.
Recall that for a closed manifold $S$, the Grassmannian $\Gr_S(M)$, i.e., the space of all submanifolds in $M$ which are diffeomorphic to $S$, is a smooth Fr\'echet manifold whose tangent space at $N\in\Gr_S(M)$ can be canonically identified as $T_N\Gr_S(M)=\Gamma(TM|_N/TN)$.
The Grassmannian is the base of a (locally trivial) smooth principal bundle 
\begin{equation}\label{E:principal}
\Emb_S(M)\to\Gr_S(M),\quad\varphi\mapsto\varphi(S),
\end{equation}
with structure group $\Diff(S)$, see \cite{BF81,M80b,M80c} and \cite[Theorem~44.1]{KM}.
Recall that the space of embeddings, $\Emb_S(M)$, is a smooth Fr\'echet manifold whose tangent space at $\varphi\in\Emb_S(M)$ can be canonically identified as $T_\varphi\Emb_S(M)=\Gamma(\varphi^*TM)$.
Moreover, the group of all diffeomorphisms, $\Diff(S)$, is a Fr\'echet Lie group with Lie algebra $\mathfrak X(S)$, the Lie algebra of vector fields.
We will show that the space of nonlinear frames, i.e., the space of all parametrized flags, is the total space of a smooth principal bundle over $\Flag_{\mathcal S}(M)$ with structure group $\Diff(S_1)\times\cdots\times\Diff(S_r)$ which generalizes the fundamental frame bundle over $\Gr_S(M)$ in \eqref{E:principal}.

In Proposition~\ref{L2} we will exhibit a reduction of structure groups that permits to regard (connected components of) $\Flag_{\mathcal S}(M)$ as the base of a principal bundle with total space $\Emb_{S_r}(M)$ and structure group $\Diff(S_r;\Sigma)$, the group of diffeomorphisms preserving a certain flag $\Sigma$ in $S_r$.

In Proposition~\ref{smallest} we will show that the manifold $\Flag_{\mathcal S}(M)$ is diffeomorphic to a twisted product of two flag manifolds of shorter lengths.
Iterating this observation, one is lead to a description of $\Flag_{\mathcal S}(M)$ as a twisted product of nonlinear Grassmannians, cf.\ Remark~\ref{R:pFlagtriv}.

In Proposition~\ref{no1} we will describe (connected components of) $\Flag_{\mathcal S}(M)$ as homogeneous spaces of $\Diff_c(M)$, the group of compactly supported diffeomorphisms.
Recall that the group $\Diff_c(M)$ is a smooth Lie group with Lie algebra $\mathfrak X_c(M)$, see \cite{M80b} and \cite[Theorem~43.1]{KM}.

Evidently, the aforementioned statements on nonlinear flag manifolds can be considered as generalizations of well known facts about diffeomorphism groups, spaces of embeddings and nonlinear Grassmannians.
Since the proofs we will provide rely crucially on these classical results (and little else), we start by summarizing them in Lemma~\ref{L:classics} below.

\subsection{Background on nonlinear Grassmannians}

A submanifold will be called \emph{splitting submanifold} if the corresponding (closed) linear subspace in a submanifold chart admits a complement, cf.\ \cite[Definition~27.11]{KM}.
A subgroup $H$ in a Lie group $G$ will be called a \emph{splitting Lie subgroup} if it is a splitting submanifold of $G$.
In this case, $H$ is a Lie group with the induced structure.

Recall that an action of a Lie group $G$ on a manifold $\mathcal M$ is said to \emph{admit local smooth sections} if, for every $x\in\mathcal M$, the map provided by the action, $G\to\mathcal M$, $g\mapsto g(x)$, admits a smooth local right inverse defined in an open neighborhood of $x$.
More explicitly, we require that for every point $x\in\mathcal M$ there exists an open neighborhood $U$ of $x$ in $\mathcal M$ and a smooth map $s\colon U\to G$ such that $s(y)(x)=y$, for all $y\in U$.
In this situation we may w.l.o.g.\ moreover assume that $s(x)$ is the neutral element in $G$.
Clearly, any action which admits local smooth sections is locally and infinitesimally transitive.
In particular, its orbits are open and closed in $\mathcal M$ and, hence, they consist of several connected components of $\mathcal M$.

\begin{lemma}\label{L:classics}
Consider a smooth manifold $M$ and a closed smooth manifold $S$.
Then, for $\varphi\in\Emb_S(M)$ and $N:=\varphi(S)\in\Gr_S(M)$, the following hold true:
\begin{enumerate}[(a)]
\item\label{I:EmbGr}
The map $\Emb_S(M)\to\Gr_S(M)$, $\varphi\mapsto\varphi(S)$, is a (locally trivial) smooth principal bundle with structure group $\Diff(S)$.
\item\label{I:DiffGr} 
The $\Diff_c(M)$ action on $\Gr_S(M)$ is smooth and admits local smooth sections.
The isotropy group $\Diff_c(M;N):=\{f\in\Diff_c(M):f(N)=N\}$ is a splitting Lie subgroup in $\Diff_c(M)$ with Lie algebra $\mathfrak X_c(M;N)=\{X\in\mathfrak X_c(M):X(N)\subseteq TN\}$.
In particular, $\Gr_S(M)_N$, the $\Diff_c(M)$ orbit through $N$, consists of several connected components of $\Gr_S(M)$ and the map provided by the action, $\Diff_c(M)\to\Gr_S(M)_N$, $f\mapsto f(N)$, is a smooth principal bundle with structure group $\Diff_c(M;N)$.
Hence,
\[
\Gr_S(M)_N=\Diff_c(M)/\Diff_c(M;N)
\]
may be regarded as a homogeneous space. 
\item\label{I:DiffEmb}
The $\Diff_c(M)$ action on $\Emb_S(M)$ is smooth and admits local smooth sections. 
The isotropy group $\Diff_c(M;\varphi):=\{f\in\Diff_c(M):f\circ\varphi=\varphi\}$ is a splitting (normal) Lie subgroup in $\Diff_c(M;\varphi(S))$ with Lie algebra $\mathfrak X_c(M;\varphi)=\{X\in\mathfrak X_c(M):X\circ\varphi=0\}$.
In particular, $\Emb_S(M)_\varphi$, the $\Diff_c(M)$ orbit through $\varphi$, consists of several connected components of $\Emb_S(M)$ and the map provided by the action, $\Diff_c(M)\to\Emb_S(M)_\varphi$, $f\mapsto f\circ\varphi$, is a smooth principal bundle with structure group $\Diff_c(M;\varphi)$.
Hence,
\[
\Emb_S(M)_\varphi=\Diff_c(M)/\Diff_c(M;\varphi)
\]
may be regarded as a homogeneous space 
\item\label{I:DiffDiff}
The $\Diff_c(M;N)$ action on $\Diff_S(N)$, the manifold of all diffeomorphisms from $S$ onto $N$, is smooth and admits local smooth sections.
In particular, $\Diff_S(N)_\varphi$, the $\Diff_c(M;N)$ orbit through $\varphi$, consists of several connected components of $\Diff_S(N)$ and the map provided by the action, $\Diff_c(M;N)\to\Diff_S(N)_\varphi$, $f\mapsto f\circ\varphi$, is a smooth principal bundle with structure group $\Diff_c(M;\varphi)$.
Hence, 
\[
\Diff_S(N)_\varphi=\Diff_c(M;N)/\Diff_c(M;\varphi)
\]
may be regarded as a homogeneous space.
\end{enumerate}
\end{lemma}

The statement in (\ref{I:EmbGr}) has been proved by Binz and Fischer \cite{BF81} for compact $S$.
The generalization to noncompact $S$ is due to Michor, see \cite{BF81,M80b,M80c} and \cite[Theorem~44.1]{KM}.
For a manifold $S$ with nonempty boundary, this bundle has been studied in \cite[Theorem~2.2]{GBV14}.
The statements in (\ref{I:DiffGr}), (\ref{I:DiffEmb}) and (\ref{I:DiffDiff}) appear to be well known among experts, see for instance \cite{MMM13}. 
For the sake of completeness we will now sketch a proof.

Let $\alpha\colon TM\to M$ be a smooth map such that $TM\to M\times M$, $X\mapsto(\pi(X),\alpha(X))$, is a tubular neighborhood of the diagonal, where $\pi\colon TM\to M$ denotes the tangent bundle projection.
In particular, we assume $\alpha(0_x)=x$ for all $x\in M$.
If $\mathcal U$ is a sufficiently $C^1$ small zero neighborhood in $\Gamma_c(TM)$, then
\[
\Gamma_c(TM)\supseteq\mathcal U\to\Diff_c(M),\quad X\mapsto\alpha\circ X,
\]
is a standard chart for the smooth structure on $\Diff_c(M)$ centered at the identity, see \cite[Theorem~43.1]{KM}.
We may choose $\alpha$ such that $X\in TN\Leftrightarrow(\pi(X),\alpha(X))\in N\times N$.
Thus, in the aforementioned chart, the sequence of subgroups 
\[
\Diff_c(M;\varphi)\subseteq\Diff_c(M;N)\subseteq\Diff_c(M)
\]
corresponds to the sequence of linear inclusions
\[
\{X\in\Gamma_c(TM):X|_N=0\}\subseteq\{X\in\Gamma_c(TM):X(N)\subseteq TN\}\subseteq\Gamma_c(TM).
\]
Since both linear inclusions admit complements, we see that $\Diff_c(M;\varphi)$ is a splitting Lie subgroup of $\Diff_c(M;N)$ and the latter is a splitting Lie subgroup of $\Diff_c(M)$.

Let $\sigma\colon\Gamma(\varphi^*TM)\to\Gamma_c(TM)$ be a smooth linear map such that $\sigma(Y)\circ\varphi=Y$, for all $Y\in\Gamma(\varphi^*TM)$.
If $\mathcal V$ is a sufficiently small zero neighborhood in $\Gamma(\varphi^*TM)$, then $\sigma(\mathcal V)\subseteq\mathcal U$ and $\Gamma(\varphi^*TM)\supseteq\mathcal V\to\Emb_S(M)$, $Y\mapsto\alpha\circ Y$, is a standard chart for the smooth structure on $\Emb_S(M)$ centered at $\varphi$, see \cite[Theorem~42.1]{KM}.
By construction, $\alpha\circ Y=(\alpha\circ\sigma(Y))\circ\varphi$.
Therefore, $\alpha\circ Y\mapsto\alpha\circ\sigma(Y)$, is a local smooth section for the $\Diff_c(M)$ action on $\Emb_S(M)$, whence (\ref{I:DiffEmb}).

Composing local smooth sections for the $\Diff_c(M)$ action on $\Emb_S(M)$ with local smooth sections of the frame bundle $\Emb_S(M)\to\Gr_S(M)$, we obtain local smooth sections for the action of $\Diff_c(M)$ on $\Gr_S(M)$, whence (\ref{I:DiffGr}).
Restricting local smooth sections for the $\Diff_c(M)$ action on $\Emb_S(M)$ along the inclusion $\Diff_S(N)\subseteq\Emb_S(M)$, we obtain local smooth sections for the action of $\Diff_c(M;N)$ on $\Diff_S(N)$, whence (\ref{I:DiffDiff}).

Note that the first assertion in Lemma~\ref{L:classics}(\ref{I:DiffEmb}) may be considered as a strengthening of the classical isotopy extension theorem \cite[Theorem~1.3 in Chapter~8]{Hirsch}.

We will also use the following simple fact.

\begin{remark}\label{R:GrSLM}
If $L$ is a closed submanifold in $M$, then $\Gr_S(L)$ is a splitting smooth submanifold in $\Gr_S(M)$.
Indeed, given $N\in\Gr_S(L)$, we may use a tubular neighborhood $\psi\colon TM|_N/TN\to M$ of $N$ in $M$ with the property $X\in TL|_N/TN\Leftrightarrow\psi(X)\in L$ to write down a local chart for $\Gr_S(M)$
centered at $N$,
\[
\Gamma(TM|_N/TN)\to\Gr_S(M),\quad\xi\mapsto\psi(\xi(N)),
\]
in which the inclusion $\Gr_S(L)\subseteq\Gr_S(M)$ corresponds to the linear inclusion $\Gamma(TL|_N/TN)\subseteq\Gamma(TM|_N/TN)$.
Clearly, $\Gamma(TL|_N/TN)$ admits a complement in $\Gamma(TM|_N/TN)$ which is isomorphic to $\Gamma(TM|_N/TL|_N)$.
\end{remark}

\subsection{The fundamental frame bundle}

Suppose $S_1,\dotsc,S_r$ are closed smooth manifolds and put $\mathcal S:=(S_1,\dotsc,S_r)$.
Let 
\[
\Flag_{\mathcal S}(M):=\Flag_{S_1,\dotsc,S_r}(M):=\left\{\N=(N_1,\dotsc,N_r)\in\prod_{i=1}^r\Gr_{S_i}(M)\middle|\forall i:N_i\subseteq N_{i+1}\right\}
\]
denote the space of all \emph{nonlinear flags of type $\mathcal S$} in a smooth, possibly noncompact manifold $M$.
The group $\Diff(M)$ acts in an obvious way from the left on $\Flag_{\mathcal S}(M)$.
Furthermore, let 
\[
\pFlag_{\mathcal S}(M):=\left\{\Ph=(\varphi_1,\dotsc,\varphi_r)\in\prod_{i=1}^r\Emb_{S_i}(M)\middle|\forall i:\varphi_i(S_i)\subseteq\varphi_{i+1}(S_{i+1})\right\}
\]
denote the space of all parametrized nonlinear flags, i.e., \emph{nonlinear frames of type $\mathcal S$ in $M$.}
Note that the group 
\[
\Diff(\mathcal S):=\prod_{i=1}^r\Diff(S_i)
\]
acts from the right on $\pFlag_{\mathcal S}(M)$ and this action commutes with the left action of $\Diff(M)$.

\begin{proposition}\label{L1}
In this situation the following hold true:
\begin{enumerate}[(a)]
\item\label{I:Flag}
$\Flag_{\mathcal S}(M)$ is a splitting smooth submanifold of\/ $\prod_{i=1}^r\Gr_{S_i}(M)$ with tangent space
\begin{equation}\label{tnf}
T_{\mathcal N}\Flag_{\mathcal S}(M)
=\left\{(\xi_1,\dotsc,\xi_r)\in\prod_{i=1}^r\Gamma(TM|_{N_i}/TN_i)\middle|\forall i:\xi_{i+1}|_{N_i}=\xi_i\textnormal{ mod } TN_{i+1}|_{N_i}\right\}
\end{equation}
at $\mathcal N=(N_1,\dotsc,N_r)\in\Flag_{\mathcal S}(M)$.
\item\label{I:pFlag}
$\pFlag_{\mathcal S}(M)$ is a splitting smooth submanifold of\/ $\prod_{i=1}^r\Emb_{S_i}(M)$ with tangent space
\begin{align*}
T_\Phi\pFlag_{\mathcal S}(M)
=\left\{(X_1,\dotsc,X_r)\in\prod_{i=1}^r\Gamma(\varphi_i^*TM)\middle|\begin{array}{c}\forall i:X_{i+1}\circ(\ph_{i+1})^{-1}\circ\ph_i=X_i\\\textnormal{ mod }\varphi_i^*\bigl(T(\varphi_{i+1}(S_{i+1}))\bigr)\end{array}\right\}
\end{align*}
at $\Phi=(\varphi_1,\dotsc,\varphi_r)\in\pFlag_{\mathcal S}(M)$.
\item\label{I:pFlagFlag}
The canonical map, 
\begin{equation}\label{frames}
\pFlag_{\mathcal S}(M)\to\Flag_{\mathcal S}(M),\quad(\varphi_1,\dotsc,\varphi_r)\mapsto\bigl(\varphi_1(S_1),\dotsc,\varphi_r(S_r)\bigr),
\end{equation} 
is a $\Diff(M)$ equivariant smooth principal fiber bundle with structure group $\Diff(\mathcal S)$.
\end{enumerate}
\end{proposition}

\begin{proof}
It is well known that $\Emb_{S_i}(M)\to\Gr_{S_i}(M)$ is a smooth principal fiber bundle with structure group $\Diff(S_i)$, see Lemma~\ref{L:classics}(\ref{I:EmbGr}).
Hence, the product of these maps, 
\begin{equation}\label{E:pEmbSi}
\prod_{i=1}^r\Emb_{S_i}(M)\to\prod_{i=1}^r\Gr_{S_i}(M),
\end{equation}
is a smooth principal fiber bundle with structure group $\prod_{i=1}^r\Diff(S_i)$.
Clearly, $\pFlag_{\mathcal S}(M)$ is the preimage of $\Flag_{\mathcal S}(M)$ under the map \eqref{E:pEmbSi}.
Therefore, it suffices to show (\ref{I:Flag}).

We will prove (\ref{I:Flag}) by induction on $r$.
Suppose $N_r\in\Gr_{S_r}(M)$.
Since the $\Diff_c(M)$ action on $\Gr_{S_k}(M)$ admits local smooth sections, see Lemma~\ref{L:classics}(\ref{I:DiffGr}), there exists an open neighborhood $U$ of $N_r$ in $\Gr_{S_r}(M)$ and a smooth map $U\to\Diff_c(M)$, $N'_r\mapsto f_{N'_r}$, such that $f_{N_r}=\id$ and $f_{N'_r}(N_r)=N'_r$ for all $N'_r\in U$.
We obtain a diffeomorphism
\begin{equation}\label{olong}
\prod_{i=1}^{r-1}\Gr_{S_i}(M)\times U\to\prod_{i=1}^{r-1}\Gr_{S_i}(M)\times U,
\quad(N'_1,\dotsc,N'_r)\mapsto\bigl(f_{N'_r}^{-1}(N'_1),\dotsc,f_{N'_r}^{-1}(N'_{r-1}),N'_r\bigr).
\end{equation}
Clearly, this diffeomorphism maps the part of $\Flag_{\mathcal S}(M)$ contained in $\prod_{i=1}^{r-1}\Gr_{S_i}(M)\times U$ onto the subset $\Flag_{S_1,\dotsc,S_{r-1}}(N_r)\times U$ of $\prod_{i=1}^{r-1}\Gr_{S_i}(M)\times U$.
By induction, $\Flag_{S_1,\dots,S_{r-1}}(N_r)$ is a splitting smooth submanifold of $\prod_{i=1}^{r-1}\Gr_{S_i}(N_r)$.
Moreover, $\Gr_{S_i}(N_r)$ is a splitting smooth submanifold of $\Gr_{S_i}(M)$ according to Remark~\ref{R:GrSLM}.
Combining these two statements, we conclude that $\Flag_{S_1,\dotsc,S_{r-1}}(N_r)\times U$ is a splitting smooth submanifold of $\prod_{i=1}^{r-1}\Gr_{S_i}(M)\times U$.
Together with \eqref{olong}, this shows that $\Flag_{\mathcal S}(M)$ is a splitting smooth submanifold in $\prod_{i=1}^r\Gr_{S_i}(M)$.
It is straightforward to track the tangent spaces through this inductive proof and establish the description in \eqref{tnf}.
\end{proof}

\begin{remark}
Note that the principal $\Diff(\mathcal S)$ bundle \eqref{frames} is the restriction of the principal bundle in \eqref{E:pEmbSi} along the inclusion $\Flag_{\mathcal S}(M)\subseteq\prod_{i=1}^r\Gr_{S_i}(M)$.
\end{remark}

\begin{remark}[Riemannian metric]\label{not}
The choice of a Riemannian metric on $M$ provides an identification of the normal bundle $TM|_{N_i}/TN_i$ with the Riemannian orthogonal bundle denoted by $TN_i^\perp\subseteq TM|_{N_i}$.
Thus, $N_i\subseteq N_{i+1}$ implies $TN_{i+1}^\perp|_{N_i}\subseteq TN_i^\perp$.
For $i<r$, we denote the orthogonal complement of $TN_i$ in $TN_{i+1}|_{N_i}$ by $TN_i^\dagger\subseteq TN_{i+1}|_{N_i}$. 
Clearly, $TN_i^\dagger\subseteq TN_i^\perp$ and we have the Riemannian orthogonal decomposition
\[
TN_i^\perp=TN_{i+1}^\perp|_{N_i}\oplus TN_{i}^\dagger,
\]
with orthogonal projections denoted by $p_i^\perp$ and $p_i^\dagger$.
Now the tangent space
\eqref{tnf} can be identified with
\[
T_{\mathcal N}\Flag_{\mathcal S}(M)\cong\left\{(\xi_1,\dotsc,\xi_r)\in\prod_{i=1}^r\Gamma(TN_i^\perp)\middle|\forall i:p_i^\perp(\xi_i)=\xi_{i+1}|_{N_i}\right\}.
\]
The only freedom for the $\xi_i$, $i< r$, is in their projections $\eta_i=p_i^\dagger(\xi_i)\in\Gamma(TN_i^\dagger)$, since $\xi_i=\xi_{i+1}|_{N_i}+\eta_i$.
Thus we get a further identification of the tangent bundle,
\begin{equation}\label{tnt}
T_{\mathcal N}\Flag_{\mathcal S}(M)
\cong\prod_{i=1}^{r-1}\Gamma(TN_i^\dagger)\times\Gamma(TN_r^\perp)
\cong\prod_{i=1}^{r-1}\Gamma(TN_{i+1}|_{N_i}/TN_i)\times\Gamma(TM|_{N_r}/TN_r).
\end{equation}
Note that these identifications are invariant under the group of isometries of $M$, but not $\Diff(M)$ invariant.
\end{remark}

\subsection{A tower of Grassmannians}\label{SS:twist}

Suppose, for a moment, that $\mathcal S=(S_1,S_2)$ consist of just two model manifolds.
Then 
\begin{equation*}\label{tower}
\Flag_{\mathcal S}(M)\to\Gr_{S_2}(M),\quad (N_1,N_2)\mapsto N_2,
\end{equation*}
is the associated bundle to the principal bundle $\Emb_{S_2}(M)\to\Gr_{S_2}(M)$ for the natural $\Diff(S_2)$ action on $\Gr_{S_1}(S_2)$.
To see this, we first observe that the projection $\pFlag_{\mathcal S}(M)\to\Emb_{S_2}(M)$, $(\varphi_1,\varphi_2)\mapsto\varphi_2$, is a trivializable fiber bundle with typical fiber $\Emb_{S_1}(S_2)$.
Indeed, the canonical identification
\[
\pFlag_{\mathcal S}(M)=\Emb_{S_2}(M)\times\Emb_{S_1}(S_2),\quad(\varphi_1,\varphi_2)\leftrightarrow(\varphi_2,\varphi_2^{-1}\circ\varphi_1),
\]
is a diffeomorphism, cf.\ the proof of Proposition~\ref{smallest}(\ref{I:pFlagppp}) below.
Via this identification, the natural right action of $\Diff(\mathcal S)=\Diff(S_1)\x\Diff(S_2)$ on $\pFlag_{\mathcal S}(M)$ becomes 
\[
(\varphi_2,\tilde\varphi)\cdot (g_1,g_2)=(\varphi_2\circ g_2,g_2^{-1}\circ\tilde\varphi\circ g_1),
\]
where $g_i\in\Diff(S_i)$, $\varphi_2\in\Emb_{S_2}(M)$, $\tilde\varphi\in\Emb_{S_1}(S_2)$.
Hence, the principal bundle projection $\pFlag_{\mathcal S}(M)\to\Flag_{\mathcal S}(M)$ factors as a composition of two principal bundles,
\[
\pFlag_{\mathcal S}(M)\xrightarrow{\Diff(S_1)}\Emb_{S_2}(M)\times\Gr_{S_1}(S_2)\xrightarrow{\Diff(S_2)}\Flag_{\mathcal S}(M),
\]
where the arrows are labeled with the structure groups.
Whence the required diffeomorphism of bundles over $\Gr_{S_2}(M)$,
\[
\Flag_{\mathcal S}(M)\cong\Emb_{S_2}(M)\times_{\Diff(S_2)}\Gr_{S_1}(S_2).
\]

Let us now formulate this observation for general $\mathcal S$.

\begin{proposition}\label{smallest}
Consider a decomposition of $\mathcal S=(S_1,\dots,S_r)$ into two shorter sequences $\mathcal S'=(S_1,\dots,S_{\ell})$ and $\mathcal S''=(S_{\ell+1},\dotsc,S_r)$ where $1\leq\ell<r$.
Then the following hold true:
\begin{enumerate}[(a)]
\item\label{I:pFlagppp}
The natural map
\begin{align}\label{E:17}
\pFlag_{\mathcal S}(M)&\to\pFlag_{\mathcal S''}(M)\times\pFlag_{\mathcal S'}(S_{\ell+1}),\\\notag
(\varphi_1,\dotsc,\varphi_r)&\mapsto\bigl(\varphi_{\ell+1},\dotsc,\varphi_r;\varphi_{\ell+1}^{-1}\circ\varphi_1,\dotsc,\varphi_{\ell+1}^{-1}\circ\varphi_\ell\bigr),
\end{align}
is a diffeomorphism.
In particular, the forgetful map 
$$
\pFlag_{\mathcal S}(M)\to\pFlag_{\mathcal S''}(M),\quad(\varphi_1,\dotsc,\varphi_r)\mapsto(\varphi_{\ell+1},\dotsc,\varphi_r\bigr),
$$ 
is a trivializable smooth fiber bundle with typical fiber $\pFlag_{\mathcal S'}(S_{\ell+1})$. 
\item\label{I:Flagppp}
The forgetful map 
\begin{equation}\label{oino}
\Flag_{\mathcal S}(M)\to\Flag_{\mathcal S''}(M),\quad(N_1,N_2,\dots,N_r)\mapsto (N_{\ell+1},\dots,N_r),
\end{equation}
is a smooth fiber bundle with typical fiber $\Flag_{\mathcal S'}(S_{\ell+1})$ which is canonically isomorphic to 
\begin{equation}\label{E:18}
\pFlag_{\mathcal S''}(M)\times_{\Diff(\mathcal S'')}\Flag_{\mathcal S'}(S_{\ell+1})
\to \Flag_{\mathcal S''}(M),
\end{equation}
the associated bundle to the principal bundle $\pFlag_{\mathcal S''}(M)\to\Flag_{\mathcal S''}(M)$ for the action of the structure group $\Diff(\mathcal S'')$ on $\Flag_{\mathcal S'}(S_{\ell+1})$ via its $\Diff(S_{\ell+1})$ component.
\item\label{I:diagppp}
These bundle maps fit into the following $\Diff(M)$ equivariant commutative diagram
\begin{equation}\label{name}
\vcenter{\xymatrix{
\pFlag_{\mathcal S}(M)\ar[rrr]^-{\pFlag_{\mathcal S'}(S_{\ell+1})}\ar[d]_{\Diff(\mathcal S)}&&&\pFlag_{\mathcal S''}(M)\ar[d]^{\Diff(\mathcal S'')}
\\
\Flag_{\mathcal S}(M)\ar[rrr]^-{\Flag_{\mathcal S'}(S_{\ell+1})}&&&\Flag_{\mathcal S''}(M),
}}
\end{equation}
where each arrow is labeled with its typical fiber or structure group, respectively. 
\end{enumerate}
\end{proposition}

\begin{proof}
Clearly, the map in \eqref{E:17} is bijective with inverse,
\begin{align*}
\pFlag_{\mathcal S''}(M)\times\pFlag_{\mathcal S'}(S_{\ell+1})&\to\pFlag_{\mathcal S}(M),\\
\bigl(\varphi_{\ell+1},\dotsc,\varphi_r;\tilde\varphi_1,\dotsc,\tilde\varphi_\ell\bigr)&\mapsto\bigl(\varphi_{\ell+1}\circ\tilde\varphi_1,\dotsc,\varphi_{\ell+1}\circ\tilde\varphi_\ell,\varphi_{\ell+1},\dotsc,\varphi_r\bigr).
\end{align*}
Smoothness of the inverse follows from the fact that this is the restriction of a smooth map $\prod_{i=\ell+1}^r\Emb_{S_i}(M)\times\prod_{i=1}^\ell\Emb_{S_i}(S_{\ell+1})\to\prod_{i=1}^r\Emb_{S_i}(M)$ given by the same formula and from Proposition~\ref{L1}(\ref{I:pFlag}).
To check smoothness of the map in \eqref{E:17}, we fix $\varphi\in\Emb_{S_{\ell+1}}(M)$.
Since the $\Diff_c(M)$ action on $\Emb_{S_{\ell+1}}(M)$ admits local smooth sections, see Lemma~\ref{L:classics}(\ref{I:DiffEmb}), there exists an open neighborhood $U$ of $\varphi$ in $\Emb_{S_{\ell+1}}(M)$ and a smooth map $f\colon U\to\Diff_c(M)$ such that $f_\varphi=\id$ and $f_{\varphi_{\ell+1}}\circ\varphi=\varphi_{\ell+1}$ for all $\varphi_{\ell+1}\in U$.
Moreover, we let $\phi\colon V\to S_{\ell+1}$ denote a smooth extension of $\varphi^{-1}\colon\varphi(S_{\ell+1})\to S_{\ell+1}$ to an open neighborhood $V$ of $\varphi(S_{\ell+1})$ in $M$.
Then the map in \eqref{E:17} may be expressed in the form
$$
\bigl(\varphi_1,\dotsc,\varphi_r\bigr)\mapsto\bigl(\varphi_{\ell+1},\dotsc,\varphi_r;\phi\circ f_{\varphi_{\ell+1}}^{-1}\circ\varphi_1,\dotsc,\phi\circ f_{\varphi_{\ell+1}}^{-1}\circ\varphi_\ell\bigr),
$$
provided $\varphi_{\ell+1}\in U$ and $\varphi_{\ell+1}(S_{\ell+1})\subseteq V$.
Note that the same formula provides a smooth extension, mapping an open subset in $\prod_{i=1}^r\Emb_{S_i}(M)$ into $\prod_{i=\ell+1}^r\Emb_{S_i}(M)\times\prod_{i=1}^\ell\Emb_{S_i}(S_{\ell+1})$.
Hence, using Proposition~\ref{L1}(\ref{I:pFlag}), we conclude that \eqref{E:17} is smooth.
This proves (\ref{I:pFlagppp}).

Using Proposition~\ref{L1}(\ref{I:pFlagFlag}) and \cite[Section~37.12]{KM}, one readily checks that \eqref{E:17} induces a diffeomorphism as indicated in \eqref{E:18}, whence (\ref{I:Flagppp}).
The statements in (\ref{I:diagppp}) are now obvious.
\end{proof}

\begin{remark}\label{R:pFlagtriv}
Iterating Proposition~\ref{smallest}(\ref{I:pFlagppp}) we obtain a canonical diffeomorphism:
\begin{align*}
\pFlag_{\mathcal S}(M)&=\Emb_{S_1}(S_2)\times\Emb_{S_2}(S_3)\times\cdots\times\Emb_{S_{r-1}}(S_r)\times\Emb_{S_r}(M)
\\
(\varphi_1,\dotsc,\varphi_r)&\mapsto\bigl(\varphi_2^{-1}\circ\varphi_1,\varphi_3^{-1}\circ\varphi_2,\dotsc,\varphi_r^{-1}\circ\varphi_{r-1},\varphi_r\bigr)
\end{align*}
Iterating Proposition~\ref{smallest}(\ref{I:Flagppp}) we see that the nonlinear flag manifold $\Flag_\mathcal S(M)$ may be regarded as a twisted product of the nonlinear Grassmannians $\Gr_{S_1}(S_2),\dotsc,\Gr_{S_{r-1}}(S_r)$ and $\Gr_{S_r}(M)$.
\end{remark}

\begin{remark}[Decorated nonlinear Grassmannians]
In the one limiting case, $\ell=r-1$, we have $\mathcal S'=(S_1,\dots,S_{r-1})$, $\mathcal S''=S_r$ and the commutative diagrams in \eqref{name} becomes:
$$
\xymatrix{
\pFlag_{\mathcal S}(M)\ar[rr]^-{\pFlag_{\mathcal S'}(S_r)}\ar[d]_{\Diff(\mathcal S)}&&\Emb_{S_r}(M)\ar[d]^{\Diff(S_r)}
\\
\Flag_{\mathcal S}(M)\ar[rr]^-{\Flag_{\mathcal S'}(S_r)}&&\Gr_{S_r}(M).
}
$$
The forgetful map \eqref{oino} becomes $(N_1,\dots,N_r)\mapsto N_r$.
This allows to interpret nonlinear flags as nonlinear Grassmannians decorated with an extra structure: the flag $(N_1,\dotsc,N_r)$ can be seen as a submanifold $N_r$ of $M$ decorated with a nonlinear flag $(N_1,\dotsc,N_{r-1})\in\Flag_{\mathcal S'}(N_r)$.
\end{remark}

\subsection{Nonlinear flag manifolds as homogeneous spaces}

Clearly, the action of $\Diff_c(M)$ on the frame bundle $\pFlag_{\mathcal S}(M)$ will in general not be locally transitive if $r>1$.
However, the action of $\Diff_c(M)$ on the flag manifold $\Flag_{\mathcal S}(M)$ is locally transitive.
More precisely, we will now show that (connected components of) $\Flag_{\mathcal S}(M)$ are homogeneous spaces of $\Diff_c(M)$.
In the subsequent section we will exhibit a reduction of structure groups for the frame bundle $\pFlag_{\mathcal S}(M)\to\Flag_{\mathcal S}(M)$ with a locally transitive $\Diff_c(M)$ action on its total space.

\begin{proposition}\label{no1}
For $\mathcal N=(N_1,\dotsc,N_r)\in\Flag_{\mathcal S}(M)$ the following hold true:
\begin{enumerate}[(a)]
\item\label{I:DiffFlag}
The $\Diff_c(M)$ action on $\Flag_{\mathcal S}(M)$ is smooth and admits local smooth sections.
In particular, this action is locally and infinitesimally transitive.
Moreover, $\Flag_{\mathcal S}(M)_{\mathcal N}$, the $\Diff_c(M)$ orbit through $\mathcal N$, consists of several connected components of $\Flag_{\mathcal S}(M)$.
\item\label{I:DiffMN}
The isotropy group, 
$$
\Diff_c(M;\mathcal N):=\Diff_c(M;N_1,\dotsc,N_r):=\left\{g\in\Diff_c(M)\middle|\forall i:g(N_i)=N_i\right\},
$$ 
is a splitting Lie subgroup in $\Diff_c(M)$ with Lie algebra 
$$
\mathfrak X_c(M;\mathcal N):=\mathfrak X_c(M;N_1,\dotsc,N_r):=\left\{X\in\mathfrak X_c(M)\middle|\forall i:X(N_i)\subseteq TN_i\right\}.
$$
\item\label{I:FlagMN}
The map provided by the action, $\Diff_c(M)\to\Flag_{\mathcal S}(M)_{\mathcal N}$, $f\mapsto\bigl(f(N_1),\dotsc,f(N_r)\bigr)$, is a smooth principal fiber bundle with structure group $\Diff_c(M;\mathcal N)$.
Hence, 
$$
\Flag_{\mathcal S}(M)_{\mathcal N}=\Diff_c(M)/\Diff_c(M;\mathcal N)
$$ 
may be regarded as a homogeneous space.
\end{enumerate}
\end{proposition}

\begin{proof}
To show (\ref{I:DiffFlag}), we proceed by induction on $r$.
Since the $\Diff_c(M)$ action on $\Gr_{S_r}(M)$ admits local sections, see Lemma~\ref{L:classics}(\ref{I:DiffGr}), there exists an open neighborhood $U$ of $N_r$ in $\Gr_{S_r}(M)$ and a smooth map $U\to\Diff_c(M)$, $N_r'\mapsto f_{N_r'}$, such that $f_{N_r}=\id$ and 
\begin{equation}\label{E:x1}
f_{N_r'}(N_r)=N_r',
\end{equation}
for all $N_r'\in U$.
Using Proposition~\ref{L1}(\ref{I:Flag}) and Remark~\ref{R:GrSLM}, we see that $\Flag_{S_1,\dotsc,S_{r-1}}(N_r)$ is a splitting smooth submanifold in $\prod_{i=1}^{r-1}\Gr_{S_i}(M)$.
Hence, we obtain a smooth map
\begin{align}\label{E:x2}
q\colon\bigl\{(N_1',\dotsc,N_r')\in\Flag_{\mathcal S}(M):N_r'\in U\bigr\}&\to\Flag_{S_1,\dotsc,S_{r-1}}(N_r),\notag\\
q(N_1',\dotsc,N_r')&:=\bigl(f_{N_r'}^{-1}(N_1'),\dotsc,f_{N_r'}^{-1}(N_{r-1}')\bigr).
\end{align}
Clearly, $q_{N_1,\dotsc,N_r}=(N_1,\dotsc,N_{r-1})$.

By the induction hypothesis, the $\Diff(N_r)$ action on $\Flag_{S_1,\dotsc,S_{r-1}}(N_r)$
admits local smooth sections. Thus
there exists an open neighborhood $V$ of $(N_1,\dotsc,N_{r-1})$ in $\Flag_{S_1,\dotsc,S_{r-1}}(N_r)$ and a smooth map $g\colon V\to\Diff(N_r)$, $(N_1',\dotsc,N_{r-1}')\mapsto g_{N_1',\dotsc,N_{r-1}'}$, such that $g_{N_1,\dotsc,N_{r-1}}=\id$ and $g_{N_1',\dotsc,N_{r-1}'}(N_1,\dotsc,N_{r-1})=(N_1',\dotsc,N_{r-1}')$ for all $(N_1',\dotsc,N_{r-1}')\in V$.
Moreover, in view of Lemma~\ref{L:classics}(\ref{I:DiffDiff}), there exists an open neighborhood $W$ of the identity in $\Diff(N_r)$ and a smooth map $h\colon W\to\Diff_c(M)$, such that $h(\id)=\id$ and $h(g)|_{N_r}=g$, for all $g\in W$.
Hence, $\tilde V:=g^{-1}(W)$ is an open neighborhood of $(N_1,\dotsc,N_{r-1})$ in $\Flag_{S_1,\dotsc,S_{r-1}}(N_r)$ and $\tilde g\colon\tilde V\to\Diff_c(M)$, $\tilde g:=h\circ g$, is a smooth map such that $\tilde g_{N_1,\dotsc,N_{r-1}}=\id$ and 
\begin{equation}\label{E:x3}
\tilde g_{N_1',\dotsc,N_{r-1}'}(N_1,\dotsc,N_{r-1},N_r)=(N_1',\dotsc,N_{r-1}',N_r),
\end{equation}
for all $(N_1',\dotsc,N_{r-1}')\in\tilde V$.

We obtain an open neighborhood $\tilde U:=q^{-1}(\tilde V)$ of $(N_1,\dotsc,N_r)$ in $\Flag_{\mathcal S}(M)$ and a smooth map
\[
k\colon\tilde U\to\Diff_c(M),\qquad(N_1,\dotsc,N_r')\mapsto k_{N_1',\dotsc,N_r'}:=f_{N_r'}\circ\tilde g_{q(N_1',\dotsc,N_r')}.
\]
Clearly, $k_{N_1,\dotsc,N_r}=\id$.
Furthermore, using the equations in \eqref{E:x1}, \eqref{E:x2} and \eqref{E:x3} one readily verifies that $k_{N_1',\dotsc,N_r'}(N_1,\dotsc,N_r)=(N_1',\dotsc,N_r')$, for all $(N_1',\dotsc,N_r')\in\tilde U$.
Hence, this $k$ is a local smooth section for the $\Diff_c(M)$ action on $\Flag_{\mathcal S}(M)$.

To show (\ref{I:DiffMN}) we proceed, again, by induction on $r$.
It is well known that $\Diff_c(M;N_r)$ is a splitting Lie subgroup in $\Diff_c(M)$, see Lemma~\ref{L:classics}(\ref{I:DiffGr}).
Moreover, restriction provides a map $p\colon\Diff_c(M;N_r)\to\Diff(N_r)$ which is a smooth principal fiber bundle after disregarding the connected components of $\Diff(N_r)$ which are not in the image, see Lemma~\ref{L:classics}(\ref{I:DiffDiff}).
By induction, $\Diff(N_r;N_1,\dotsc,N_{r-1})$ is a splitting Lie subgroup in $\Diff(N_r)$.
Using the obvious relation 
\[
\Diff_c(M;N_1,\dotsc,N_r)=p^{-1}\bigl(\Diff(N_r;N_1,\dotsc,N_{r-1})\bigr),
\] 
we see that $\Diff_c(M;N_1,\dotsc,N_r)$ is a splitting Lie subgroup in $\Diff_c(M;N_r)$.
Since the latter is a splitting Lie subgroup in $\Diff_c(M)$, we conclude that $\Diff_c(M;N_1,\dotsc,N_r)$ is a splitting Lie subgroup in $\Diff_c(M)$.

The statement in (\ref{I:FlagMN}) is an immediate consequence of (\ref{I:DiffFlag}) and (\ref{I:DiffMN}).
\end{proof}

\subsection{A reduction of structure groups}\label{2.2}

Consider a sequence of embeddings 
\begin{equation}\label{star}
S_1\xrightarrow{\iota_1}S_2\xrightarrow{\iota_2}S_3\to\cdots\to S_{r-1}\xrightarrow{\iota_{r-1}}S_r
\end{equation} 
and put $\iota:=(\iota_1,\dotsc,\iota_{r-1})$.

Denote the subset of all frames in $\pFlag_{\mathcal S}(M)$ which are compatible with this sequence by
$$
\pFlag_{\mathcal S,\iota}(M):=\left\{(\varphi_1,\dotsc,\varphi_r)\in\prod_{i=1}^r\Emb_{S_i}(M)\middle|\forall i:\varphi_{i+1}\circ\iota_i=\varphi_i\right\}.
$$
Projecting out the last component provides a canonical identification
\begin{equation}\label{fe}
\pFlag_{\mathcal S,\iota}(M)=\Emb_{S_r}(M),\quad(\varphi_1,\dotsc,\varphi_r)\mapsto\varphi_r,
\end{equation}
the other embeddings can be recovered from $\varphi_r$ via $\varphi_i=\varphi_r\circ\iota_{r-1}\circ\dotsc\circ\iota_i$.

Moreover, let
\begin{equation}\label{grup}
\Diff(\mathcal S;\iota):=\left\{(g_1,\dotsc,g_r)\in\prod_{i=1}^r\Diff(S_i)\middle|\forall i:g_{i+1}\circ\iota_i=\iota_i\circ g_i\right\}
\end{equation}
denote the subgroup of all diffeomorphisms in $\Diff(\mathcal S)$ which are compatible with the sequence in \eqref{star}.
Clearly, $\pFlag_{\mathcal S,\iota}(M)$ is invariant under the action of $\Diff(\mathcal S;\iota)$.
Projecting out the last component, we obtain a canonical identification
\begin{equation}\label{E:DiffSrSigma}
\Diff(\mathcal S;\iota)=\Diff(S_r;\Sigma),\quad(g_1,\dotsc,g_r)\mapsto g_r,
\end{equation}
with the isotropy group of $\Sigma:=(\Sigma_1,\dotsc,\Sigma_{r-1})\in\Flag_{S_1,\dotsc,S_{r-1}}(S_r)$, where
\begin{equation}\label{E:Sigmai}
\Sigma_i:=(\iota_{r-1}\circ\cdots\circ\iota_i)(S_i).
\end{equation}
The other diffeomorphisms can be recovered from $g_r$ via $g_i=(\iota_{r-1}\circ\cdots\circ\iota_i)^{-1}\circ g_r\circ(\iota_{r-1}\circ\cdots\iota_i)$.

Finally, let $\Flag_{\mathcal S,\iota}(M)$ denote the image of $\pFlag_{\mathcal S,\iota}(M)$ under the map $\pFlag_{\mathcal S}(M)\to\Flag_{\mathcal S}(M)$ in \eqref{frames}.
Using the canonical identification \eqref{fe}, this can equivalently be characterized by
$$
\Flag_{\mathcal S,\iota}(M)=\left\{(N_1,\dotsc,N_r)\in\prod_{i=1}^r\Gr_{S_i}(M)\middle|\exists\varphi_r\in\Emb_{S_r}(M):\forall i:N_i=\varphi_r(\Sigma_i)\right\}.
$$ 
This will be referred to as the space of \emph{nonlinear flags of type $(\mathcal S,\iota)$ in $M$.}

\begin{proposition}\label{L2}
With this notation the following hold true:
\begin{enumerate}[(a)]
\item\label{I:Diffiota}
$\Diff(\mathcal S;\iota)$ is a splitting Lie subgroup of $\Diff(\mathcal S)$ with Lie algebra
\begin{equation}\label{88}
\mathfrak X(\mathcal S;\iota)=\left\{(Z_1,\dotsc,Z_r)\in\prod_{i=1}^r\mathfrak X(S_i)\middle|\forall i:Z_{i+1}\circ\iota_i=T\iota_i\circ Z_i\right\}.
\end{equation}
Moreover, the canonical identification in \eqref{E:DiffSrSigma} is a diffeomorphism of Lie groups.
\item\label{I:pFlagiota}
$\pFlag_{\mathcal S,\iota}(M)$ is a splitting smooth submanifold of $\pFlag_{\mathcal S}(M)$ with tangent space
$$
T_\Phi\pFlag_{\mathcal S,\iota}(M)=\left\{(X_1,\dotsc,X_r)\in\prod_{i=1}^r\Gamma(\varphi_i^*TM)\middle|\forall i:X_{i+1}\circ\iota_i=X_i\right\}
$$
at $\Phi=(\varphi_1,\dotsc,\varphi_r)\in\pFlag_{\mathcal S,\iota}(M)$.
Moreover, the canonical identification in \eqref{fe} is a diffeomorphism which is equivariant over the isomorphism of groups in \eqref{E:DiffSrSigma}.
\item\label{I:Flagiota}
$\Flag_{\mathcal S,\iota}(M)$ is a $\Diff(M)$ invariant open and closed subset of $\Flag_{\mathcal S}(M)$. 
\item\label{I:pFlagFlagiota}
The restriction of the canonical map in \eqref{frames}, 
\begin{equation}\label{framesi}
\pFlag_{\mathcal S,\iota}(M)\to\Flag_{\mathcal S,\iota}(M),
\quad(\varphi_1,\dotsc,\varphi_r)\mapsto\bigl(\varphi_1(S_1),\dotsc,\varphi_r(S_r)\bigr),
\end{equation} 
is a $\Diff(M)$ equivariant smooth principal fiber bundle with structure group $\Diff(\mathcal S;\iota)$.
\item\label{I:diagiota}
These maps fit into the following $\Diff(M)$ equivariant commutative diagram
$$
\xymatrix{
\Emb_{S_r}(M)\ar@/_3ex/[dr]_(.3){\Diff(S_r;\Sigma)}\ar@{=}[r]&\pFlag_{\mathcal S,\iota}(M)\ar[d]^-{\Diff(\mathcal S;\iota)}\ar@{^(->}[r]&\pFlag_{\mathcal S}(M)\ar[d]^{\Diff(\mathcal S)}
\\
&\Flag_{\mathcal S,\iota}(M)\ar@{^(->}[r]&\Flag_{\mathcal S}(M),
}
$$
where the arrows indicating principal bundles are labeled with their structure groups.
Hence, this may be regarded as a reduction of the structure group along the inclusion $\Diff(S_r;\Sigma)=\Diff(\mathcal S;\iota)\subseteq\Diff(\mathcal S)$.
\end{enumerate}
\end{proposition}

\begin{proof}
Recall from Proposition~\ref{no1}(\ref{I:DiffMN}) that $\Diff(S_r;\Sigma)$ is a splitting Lie subgroup in $\Diff(S_r)$.
Using Lemma~\ref{L:classics}(\ref{I:DiffDiff}), we see that the map 
$$
\Diff(S_r;\Sigma)\to\Diff(S_i),\quad g_r\mapsto g_i:=(\iota_{r-1}\circ\cdots\circ\iota_i)^{-1}\circ g_r\circ(\iota_{r-1}\circ\cdots\iota_i),
$$
is smooth.
Hence, $\Diff(\mathcal S;\iota)$ is the graph of the smooth map 
$$
\Diff(S_r)\supseteq\Diff(S_r;\Sigma)\to\prod_{i=1}^{r-1}\Diff(S_i),\quad g_r\mapsto(g_1,\dotsc,g_{r-1}).
$$ 
We conclude that $\Diff(\mathcal S;\iota)$ is a splitting smooth submanifold in $\Diff(\mathcal S)$ and that the isomorphism of groups in \eqref{E:DiffSrSigma} is a diffeomorphism.
This shows (\ref{I:Diffiota}).

To see (\ref{I:pFlagiota}), it suffices to observe that the diffeomorphism in Remark~\ref{R:pFlagtriv} maps $\pFlag_{\mathcal S,\iota}(M)$ onto the subset
$\{\iota_1\}\times\cdots\times\{\iota_{r-1}\}\times\Emb_{S_r}(M)$.

The statement in (\ref{I:Flagiota}) is an immediate consequence of Proposition~\ref{no1}(\ref{I:DiffFlag}).

To see (\ref{I:pFlagFlagiota}), it remains to construct local sections of the map in \eqref{framesi}.
Given $\mathcal N=(N_1,\dotsc,N_r)$ in $\Flag_{\mathcal S,\iota}(M)$, there exists $(\varphi_1,\dotsc,\varphi_r)$ in $\pFlag_{\mathcal S,\iota}(M)$ such that $\varphi_i(S_i)=N_i$.
Using Proposition~\ref{no1}(\ref{I:DiffFlag}), we find an open neighborhood $U$ of $\mathcal N$ in $\Flag_{\mathcal S,\iota}(M)$ and a smooth map $f\colon U\to\Diff_c(M)$, $\mathcal N'\mapsto f_{\mathcal N'}$, such that $f_{\mathcal N}=\id$ and $f_{\mathcal N'}(\mathcal N)=\mathcal N'$ for all $\mathcal N'\in U$.
Hence, 
$$
U\to\pFlag_{\mathcal S,\iota}(M),\quad\mathcal N'\mapsto\bigl(f_{\mathcal N'}\circ\varphi_1,\dotsc,f_{\mathcal N'}\circ\varphi_r\bigr),
$$
is a local smooth section of \eqref{framesi}, mapping $\mathcal N$ to $(\varphi_1,\dotsc,\varphi_r)$.

The statements in (\ref{I:diagiota}) are now obvious.
\end{proof}

\begin{remark}
As in Proposition~\ref{smallest}, we split $\mathcal S$ into $\mathcal S'=(S_1,\dots,S_{\ell})$ and $\mathcal S''=(S_{\ell+1},\dotsc,S_r)$, with $\iota':=(\iota_1,\dotsc,\iota_{\ell-1})$ and $\iota'':=(\iota_{\ell+1},\dotsc,\iota_{r-1})$.
Moreover, we consider the flags $\Sigma:=(\Sigma_1,\dotsc,\Sigma_{r-1})$ and $\Sigma'':=(\Sigma_{\ell+1},\dotsc,\Sigma_{r-1})$ in $S_r$ with $\Sigma_i$ as in \eqref{E:Sigmai}.
As in the proof of Proposition~\ref{no1}(\ref{I:DiffMN}) one can show that the canonical homomorphism 
\[
\Diff(S_r;\Sigma)=\Diff(\mathcal S;\iota)\to\Diff(\mathcal S'';\iota'')=\Diff(S_r;\Sigma'')
\] 
is the embedding of a splitting Lie subgroup, see also Proposition~\ref{L2}(\ref{I:Diffiota}).
The canonical map $\pFlag_{\mathcal S,\iota}(M)\to\pFlag_{\mathcal S'',\iota''}(M)$ is a diffeomorphism which is equivariant over the latter homomorphism in view of Proposition~\ref{L2}(\ref{I:pFlagiota}).
The canonical map $\Flag_{\mathcal S,\iota}(M)\to\Flag_{\mathcal S'',\iota''}(M)$ that forgets the first $\ell$ submanifolds of a nonlinear flag is a smooth fiber bundle which is canonically isomorphic to the associated bundle $\pFlag_{\mathcal S'',\iota''}(M)\times_{\Diff(\mathcal S'',\iota'')}\mathcal F$, where 
\[
\mathcal F=\frac{\Diff(S_r;\Sigma'')}{\Diff(S_r;\Sigma)}=\frac{\Diff(\mathcal S'';\iota'')}{\Diff(\mathcal S;\iota)}
\] 
denotes the open and closed orbit of the nonlinear flag $\Sigma':=(\Sigma_1',\dotsc,\Sigma_\ell')$ in $S_{\ell+1}$, with $\Sigma_i'=(\iota_\ell\circ\cdots\circ\iota_i)(S_i)$, under the action of $\Diff(\mathcal S'';\iota'')$ on $\Flag_{\mathcal S'}(S_{\ell+1})$ through its $(\ell+1)$-th component.
Hence, $\mathcal F$ consists of several connected components of the $\Diff(S_{\ell+1})$ orbit $\Flag_{\mathcal S'}(S_{\ell+1})_{\Sigma'}$.
We summarize these observations in the following $\Diff(M)$ equivariant commutative diagram
\begin{equation}\label{nam}
\vcenter{
\xymatrix{
\Emb_{S_r}(M)\ar@{=}[r]\ar@/_3ex/[dr]_(.3){\Diff(S_r;\Sigma)}
&\pFlag_{\mathcal S,\iota}(M)\ar@{=}[rrr]\ar[d]_-{\Diff(\mathcal S,\iota)}
&&&\pFlag_{\mathcal S'',\iota''}(M)\ar[d]^-{\Diff(\mathcal S'',\iota'')}\ar@{=}[r]
&\Emb_{S_r}(M)\ar@/^3ex/[dl]^(.3){\Diff(S_r;\Sigma'')}
\\
&\Flag_{\mathcal S,\iota}(M)\ar[rrr]^-{\mathcal F}&&&\Flag_{\mathcal S'',\iota''}(M
}}
\end{equation}
where each arrow is labeled with its typical fiber or structure group, respectively.
\end{remark}

\subsection{Tautological bundles}\label{SS:taut}

Recall the tautological bundle over the nonlinear Grassmannian, 
\[
\mathcal T:=\{(N,x)\in\Gr_S(M)\times M:x\in N\},
\]
a splitting submanifold of $\Gr_S(M)\times M$.
The projection on the first factor $p:\T\to\Gr_S(M)$ is a bundle with typical fiber $S$, called the \emph{tautological bundle}.
It is canonically diffeomorphic to the associated bundle 
\[
\Emb_S(M)\x_{\Diff(S)}S\to \Gr_S(M)
\]
via the diffeomorphism $[\ph,x]\leftrightarrow(\ph(S),\ph(x))$.
This can be used to show that the pullback of $\mathcal T$ along the principal bundle projection $\pi\colon\Emb_S(M)\to\Gr_S(M)$ is canonically trivial, i.e., $\pi^*\mathcal T=\Emb_S(M)\times S$.
Indeed, the principal bundle $\pi^*\Emb_S(M)$ is trivial since it admits a canonical section induced by $\pi$.

Over $\Gr_S(M)_N$, the $\Diff_c(M)$ orbit of $N\in\Gr_S(M)$, the restriction of the tautological bundle is canonically diffeomorphic to the associated bundle
\[
\Diff_c(M)\x_{\Diff_c(M;N)}N\to\Gr_S(M)_N
\]
via the diffeomorphism $[f,x]\leftrightarrow(f(N),f(x))$.
All these facts appear to be well known folklore.
More general results for flag manifolds will be formulated and proved below, see Proposition~\ref{P:taut}.

Tautological bundles will be used in Section~\ref{S:oFlag} to describe transgression of differential forms.
In \cite{DJNV} they are used for the transgression of differential characters to nonlinear Grassmannians.

Over the manifold $\Flag_{\mathcal S}(M)$ of nonlinear flags we have a nested sequence of tautological bundles with typical fibers $S_1,\dotsc,S_r$.
The proof we will present below uses the description of the nonlinear flag manifold as a homogeneous space in Proposition~\ref{no1}.

\begin{proposition}[Tautological bundles]\label{P:taut}
For $1\leq i\leq r$ consider
\begin{equation*}
\mathcal T_i:=\{(N_1,\dotsc,N_r;x)\in\Flag_{\mathcal S}(M)\times M:x\in N_i\}.
\end{equation*}
Then the following hold true:
\begin{enumerate}[(a)]
\item\label{I:taut:nest} 
$\mathcal T_1\subseteq\mathcal T_2\subseteq\cdots\subseteq\mathcal T_r\subseteq\Flag_{\mathcal S}(M)\times M$ is a sequence of splitting smooth submanifolds.
\item\label{I:taut:TiT} 
The canonical projection $\mathcal T_i\to\Flag_{\mathcal S}(M)$ is a smooth fiber bundle with typical fiber $S_i$ which is canonically diffeomorphic to the pullback of the tautological bundle over $\Gr_{S_i}(M)$ along the map $\Flag_{\mathcal S}(M)\to\Gr_{S_i}(M)$, $(N_1,\dotsc,N_r)\mapsto N_i$.
\item\label{I:taut:DiffS}
We have a canonical diffeomorphism of fiber bundles over $\Flag_{\mathcal S}(M)$,
\[
\pFlag_{\mathcal S}(M)\times_{\Diff(\mathcal S)}S_i=\mathcal T_i,\quad[(\varphi_1,\dotsc,\varphi_r),x]\leftrightarrow\bigl(\varphi_1(S_1),\dotsc,\varphi_r(S_r);\varphi_i(x)\bigr),
\]
where the left hand side denotes the bundle associated \cite[Section~37.12]{KM} to the principal bundle $\pFlag_{\mathcal S}(M)\to\Flag_{\mathcal S}(M)$ and the canonical action of the structure group $\Diff(\mathcal S)$ on $S_i$ via its $i$-th component.
\item\label{I:taut:iota}
For $\iota=(\io_1,\dotsc,\io_{r-1})$ as in Section~\ref{2.2} and $1\leq i\leq r$ we have a canonical diffeomorphism of fiber bundles over $\Flag_{\mathcal S,\iota}(M)$, 
\[
\pFlag_{\mathcal S,\iota}(M)\times_{\Diff(\mathcal S;\iota)}S_i=\mathcal T_i|_{\Flag_{\mathcal S,\iota}(M)},\quad[(\varphi_1,\dotsc,\varphi_r),x]\leftrightarrow\bigl(\varphi_1(S_1),\dotsc,\varphi_r(S_r);\varphi_i(x)\bigr),
\]
where the left hand side denotes the bundle associated to the principal bundle $\pFlag_{\mathcal S,\iota}(M)\to\Flag_{\mathcal S,\iota}(M)$ and the canonical action of its structure group $\Diff(\mathcal S;\iota)$ on $S_i$.
\item\label{I:taut:DiffM}
For $\mathcal N=(N_1,\dotsc,N_r)\in\Flag_{\mathcal S}(M)$ we have a canonical diffeomorphism of fiber bundles over $\Flag_{\mathcal S}(M)_{\mathcal N}$,
\begin{equation}\label{E:taut}
\Diff_c(M)\times_{\Diff_c(M;\mathcal N)}N_i=\mathcal T_i|_{\Flag_{\mathcal S}(M)_{\mathcal N}},\quad[f,x]\leftrightarrow(f(\mathcal N),f(x)),
\end{equation}
where the left hand side denotes the bundle associated to the principal bundle $\Diff_c(M)\to\Flag_{\mathcal S}(M)_{\mathcal N}$ and the canonical action of its structure group $\Diff_c(M;\mathcal N)$ on $N_i$. 
\end{enumerate}
\end{proposition}

\begin{proof}
Fix $\mathcal N=(N_1,\dotsc,N_r)\in\Flag_{\mathcal S}(M)$.
Using Proposition~\ref{no1} one readily shows that the action $\Diff_c(M)\times M\to M$ induces a canonical diffeomorphism of bundles over $\Flag_{\mathcal S}(M)_{\mathcal N}$,
\[
\Diff_c(M)\times_{\Diff_c(M;\mathcal N)}M=\Flag_{\mathcal S}(M)_{\mathcal N}\times M,\qquad[f,x]\leftrightarrow(f(\mathcal N),f(x)).
\]
Here the left hand side denotes the bundle associated \cite[Section~37.12]{KM} to the principal bundle $\Diff_c(M)\to\Flag_{\mathcal S}(M)_{\mathcal N}$ from Proposition \ref{no1}(\ref{I:FlagMN}) and the canonical action of its structure group $\Diff_c(M;\mathcal N)$ on $M$.
For $1\leq i\leq r$, this diffeomorphism restricts to the bijection in \eqref{E:taut}.
Since the left hand side is a splitting smooth submanifold in $\Diff_c(M)\times_{\Diff_c(M;\mathcal N)}M$, the right hand side is a splitting smooth submanifold in $\Flag_{\mathcal S}(M)_{\mathcal N}\times M$.
Analogously, we see that $\mathcal T_i|_{\Flag_{\mathcal S}(M)_{\mathcal N}}$ is a splitting smooth submanifold in $\mathcal T_{i+1}|_{\Flag_{\mathcal S}(M)_{\mathcal N}}$.
As every connected component of $\Flag_{\mathcal S}(M)$ is contained in $\Flag_{\mathcal S}(M)_{\mathcal N}$, for a suitable flag $\mathcal N$, we obtain (\ref{I:taut:nest}), (\ref{I:taut:TiT}) and (\ref{I:taut:DiffM}).
Using Proposition~\ref{L1} and the description in \eqref{E:taut}, one readily checks that the bijection in (\ref{I:taut:DiffS}) is indeed a diffeomorphism.
Combining this with Proposition~\ref{L2}, we obtain (\ref{I:taut:iota}).
\end{proof}

\section{Orientations}\label{S:oFlag}

The results on nonlinear flag manifolds presented in Section~\ref{S:Flag} admit obvious oriented analogues which are important for integration. 
Oriented flags are flags equipped with orientations and may be considered as decorated flags.
The manifold of all oriented nonlinear flags, denoted $\Flag_{\mathcal S}^\ori(M)$, is a finite covering of the corresponding nonoriented counterpart $\Flag_{\mathcal S}(M)$.

Before sketching the aforementioned results for oriented nonlinear flags, we briefly recall the corresponding facts for oriented nonlinear Grassmannians.
In a short interlude we describe, via integration, a $\Diff_c(M)$ equivariant smooth injective immersion of $\Flag_{\mathcal S}^\ori(M)$ into the space of currents on $M$.
The last two subsections are dedicated to the transgression of differential forms.
We use integration along the fiber of tautological bundles to get differential forms on oriented nonlinear Grassmannians, as well as on manifolds of oriented nonlinear flags, from  differential forms on $M$.

\subsection{Oriented nonlinear Grassmannians}\label{SS:oGr}

For a manifold $N$ we let $\mathcal O_N$ denote its orientation bundle.
Hence $\Gamma(\mathcal O_N)$ is the set of orientations of $N$, and $\Diff(N)$ acts naturally on $\Gamma(\mathcal O_N)$.

The nonlinear Grassmannian of oriented submanifolds,
\[
\Gr_S^\ori(M):=\left\{(N,o):N\in\Gr_S(M),o\in\Gamma(\mathcal O_N)\right\},
\] 
is a finite covering of the nonlinear Grassmannian $\Gr_S(M)$ which is canonically diffeomorphic to an associated bundle,
\begin{equation}\label{forg}
\Gr_S^\ori(M)=\Emb_S(M)\times_{\Diff(S)}\Gamma(\mathcal O_S)\to\Gr_S(M).
\end{equation}
If $S$ is not orientable, then the typical fiber $\Gamma(\mathcal O_S)$ is empty.
Otherwise, the covering has $2^{b_0(S)}$ sheets, where $b_0(S)$ denotes the number of connected components of $S$.
In particular, this is a double covering if $S$ is connected and orientable.

Connected components of $\Gr_S^\ori(M)$ may be regarded as a homogeneous space,
\[
\Gr_S^\ori(M)_{(N,o)}=\Diff_c(M)/\Diff_c(M;N,o).
\]
Here the left hand side denotes the $\Diff_c(M)$ orbit through $(N,o)\in\Gr_S^\ori(M)$, which is an open and closed subset in $\Gr_S^\ori(M)$.
Moreover, $\Diff_c(M;N,o)$ denotes the group of all compactly supported diffeomorphisms which preserve the submanifold $N$ and its orientation, $o$.

The covering provided by the forgetful map \eqref{forg} is nontrivial over $\Gr_S(M)_N$ if and only if there exists an orientation $o$ of $N$ and a compactly supported diffeomorphism in the connected component of the identity, $\Diff_c(M)_\circ$, which preserves the submanifold $N$ but does not preserve the orientation $o$.

All this follows readily from Lemma~\ref{L:classics}.

\begin{example}
The double coverings $\Gr^\ori_{S^1}(\mathbb R^3)\to\Gr_{S^1}(\mathbb R^3)$ and $\Gr_{S^1}^\ori(S^2)\to\Gr_{S^1}(S^2)$ are nontrivial, while $\Gr_{S^1}^\ori(\mathbb R^2)\to\Gr_{S^1}(\mathbb R^2)$ and $\Gr_{S^1}^\ori(S^1\x S^1)\to\Gr_{S^1}(S^1\x S^1)$ are trivial double coverings.
Indeed, if $S^1\cong N\subseteq\mathbb R^2$ is an embedded circle, then every diffeomorphism in $\Diff_c(\mathbb R^2;N)\cap\Diff_c(\mathbb R^2)_\circ$ restricts to an orientation preserving diffeomorphism on either connected component of the complement, $\mathbb R^2\setminus N$, and, thus, preserves the (induced boundary) orientation on $N$ too.
The same argument works for contractible circles in the torus, for the complement of such a circle consists of two nondiffeomorphic connected components.
If $S^1\cong N\subseteq S^1\times S^1$ is not contractible, then the inclusion induces an injective homomorphism in first homology, $H_1(N)\to H_1(S^1\times S^1)$.
As every diffeomorphism in $\Diff(S^1\times S^1;N)\cap\Diff(S^1\times S^1)_\circ$ induces the identity on $H_1(S^1\times S^1)$, its restriction to $N$ preserves the fundamental class of $N$ and, thus, the corresponding orientation also.
\end{example}

Each orientation $o_S$ of $S$, provides a $\Diff_c(M)$ equivariant map
\begin{equation}\label{pior}
\Emb_S(M)\to\Gr_{S,o_S}^\ori(M),\quad\varphi\mapsto(\varphi(S),\varphi_*o_S),
\end{equation}
which is a principal fiber bundle with structure group $\Diff(S;o_S)$, the group of orientation preserving diffeomorphisms.
Here $\Gr_{S,o_S}^\ori(M)$ denotes the image of this map which, by equivariance, is a $\Diff_c(M)$ invariant subset in $\Gr^\ori_S(M)$, i.e., the union of several connected components, cf.\ Lemma~\ref{L:classics}(\ref{I:DiffGr}).
This subset coincides with $\Gr_S^\ori(M)$ if and only if $\Diff(S)$ acts transitively on $\Gamma(\mathcal O_S)$, that is, iff each connected component of $S$ admits an orientation reversing diffeomorphism.

\subsection{Oriented nonlinear flags}\label{SS:oFlag}

Let us denote the space of all \emph{oriented nonlinear flags of type $\mathcal S$} by
\[
\Flag^\ori_{\mathcal S}(M):=\left\{\bigl((N_1,o_1),\dotsc,(N_r,o_r)\bigr)\in\prod_{i=1}^r\Gr_{S_i}^\ori(M)\middle|\forall i:N_i\subseteq N_{i+1}\right\}.
\]
It follows from Proposition~\ref{L1}(\ref{I:Flag}) that this is a splitting smooth submanifold in $\prod_{i=1}^r\Gr^\ori_{S_i}(M)$.
Moreover, the forgetful map $\Flag_{\mathcal S}^\ori(M)\to\Flag_{\mathcal S}(M)$ is a finite covering which is canonically diffeomorphic to an associated bundle,
\[
\Flag_{\mathcal S}^\ori(M)=\pFlag_{\mathcal S}(M)\times_{\Diff(\mathcal S)}\bigl(\Gamma(\mathcal O_{S_1})\times\cdots\times\Gamma(\mathcal O_{S_r})\bigr)\to\Flag_{\mathcal S}(M),
\]
where $\mathcal O_{S_i}$ denotes the orientation bundle of $S_i$.

Proceeding as in the proof of Proposition~\ref{smallest}, one readily verifies that $\Flag_{\mathcal S}^\ori(M)$ is diffeomorphic to a twisted product of the oriented nonlinear Grassmannians $\Gr_{S_1}^\ori(S_2),\dotsc,\Gr_{S_{r-1}}^\ori(S_r)$ and $\Gr_{S_r}^\ori(M)$, cf.\ Remark~\ref{R:pFlagtriv}.

\begin{remark}[Oriented nonlinear flags of codimension one]
Suppose $M$ comes equipped with a Riemannian metric and an orientation.
If, moreover, the dimensions of the modeling manifolds form consecutive integers, i.e., if
\[
\dim(S_i)+1=\dim(S_{i+1})\quad\text{and}\quad\dim(S_r)+1=\dim(M),
\]
then the tangent bundle of the oriented nonlinear flag manifold may be described more explicitly.
Indeed, the Riemannian metric and the orientations provide trivializations of the normal bundles, $TN_{i+1}|_{N_i}/TN_i\cong N_i\times\mathbb R$ and $TM|_{N_r}/TN_r\cong N_r\times\mathbb R$.
Combining this with \eqref{tnt}, we obtain an isomorphism
\begin{equation*}\label{ftnt}
T_{(\mathcal N,o)}\Flag^\ori_{\mathcal S}(M)\cong\prod_{i=1}^rC^\infty(N_i)
\end{equation*}
at $(\mathcal N,o)\in\Flag_{\mathcal S}^\ori(M)$.
This kind of description of the tangent space is useful for the shape space of oriented nonlinear flags of curves on surfaces in $\mathbb R^3$ considered in \cite{TV}.
\end{remark}

Using Proposition~\ref{no1} we see that the $\Diff_c(M)$ action on $\Flag_{\mathcal S}^\ori(M)$ admits local smooth sections and (connected components of) $\Flag_{\mathcal S}^\ori(M)$ is a homogeneous space of $\Diff_c(M)$,
\[
\Flag_{\mathcal S}^\ori(M)_{(\mathcal N,o)}=\Diff_c(M)/\Diff_c(M;\mathcal N,o).
\]
Here the left hand side denotes the $\Diff_c(M)$ orbit through $(\mathcal N,o)=\bigl((N_1,o_1),\dotsc,(N_r,o_r)\bigr)$ in $\Flag^\ori_{\mathcal S}(M)$ which is an open and closed subset in $\Flag_{\mathcal S}^\ori(M)$ in view of Proposition~\ref{no1}(\ref{I:DiffFlag}).
Moreover, $\Diff_c(M;\mathcal N,o)$ denotes the group of all compactly supported diffeomorphisms preserving each submanifold $N_i$ and its orientation $o_i$.
Since this is an open and closed subgroup in $\Diff_c(M;\mathcal N)$ it also is a splitting Lie subgroup of $\Diff_c(M)$ in view of Proposition~\ref{no1}(\ref{I:DiffMN}).

A sequence of orientations $o_{S_i}$ on each $S_i$, denoted by $o_{\mathcal S}=(o_{S_1},\dotsc,o_{S_r})$, provides a lift of the map $\pFlag_{\mathcal S}(M)\to\Flag_{\mathcal S}(M)$ across the covering $\Flag_{\mathcal S}^\ori(M)\to\Flag_{\mathcal S}(M)$. 
Its image consists of several connected components of $\Flag^\ori_{\mathcal S}(M)$ and will be denoted by $\Flag^\ori_{\mathcal S,o_{\mathcal S}}(M)$.
The lifted map yields a smooth principal bundle
\begin{equation}\label{ppuls}
\pFlag_{\mathcal S}(M)\to\Flag^\ori_{\mathcal S,o_{\mathcal S}}(M),\quad
\bigl(\varphi_1,\dotsc,\varphi_r)\mapsto\bigl((\varphi_1(S_1),(\varphi_1)_*o_{S_1}),\dotsc,(\varphi_r(S_r),(\varphi_r)_*o_{S_r})\bigr),
\end{equation}
with structure group $\Diff(\mathcal S;o_{\mathcal S}):=\prod_{i=1}^r\Diff(S_i,o_{S_i})$.

Suppose we are given a sequence $\iota$ of embeddings as in \eqref{star}.
Then composition of \eqref{ppuls} with the inclusion $\pFlag_{\mathcal S,\iota}(M)\subseteq\pFlag_{\mathcal S}(M)$ yields a lift of the map $\Emb_{S_r}(M)=\pFlag_{\mathcal S,\iota}(M)\to\Flag_{\mathcal S,\iota}(M)\subseteq\Flag_{\mathcal S}(M)$ across the covering $\Flag_{\mathcal S}^\ori(M)\to\Flag_{\mathcal S}(M)$.
Its image consists of several connected components of $\Flag^\ori_{\mathcal S}(M)$ and will be denoted by $\Flag^\ori_{\mathcal S,\iota,o_{\mathcal S}}(M)$.
The lifted map provides a smooth principal bundle
$$
\Emb_{S_r}(M)=\pFlag_{\mathcal S,\iota}(M)\to\Flag^\ori_{\mathcal S,\iota,o_{\mathcal S}}(M)
$$
with structure group $\Diff(S_r;o_{S_r},\Sigma,o_\Sigma)=\Diff(\mathcal S;\iota,o_{\mathcal S})$.
Here $\Diff(\mathcal S;\iota,o_{\mathcal S})$ denotes the (open and closed) subgroup of all elements in $\Diff(\mathcal S;\iota)$ which preserve the orientations $o_{S_1},\dotsc,o_{S_r}$.
Hence, $\Diff(\mathcal S;\iota,o_{\mathcal S})$ is a splitting Lie subgroup of $\Diff(\mathcal S)$ in view of Proposition~\ref{L2}(\ref{I:Diffiota}).
Moreover, $\Diff(S_r;o_{S_r},\Sigma,o_\Sigma)$ denotes the (open and closed) subgroup of all diffeomorphisms in $\Diff(S_r;\Sigma)$ that preserve the orientation $o_{S_r}$ and the orientation of $\Sigma_i$ corresponding to $o_{S_i}$, cf.~\eqref{E:Sigmai}, for $i<r$.
Hence, $\Diff(S_r;o_{S_r},\Sigma,o_\Sigma)$ is a splitting Lie subgroup of $\Diff(S_r)$ in view of Proposition~\ref{no1}(\ref{I:DiffMN}).
We obtain the following $\Diff(M)$ equivariant commutative diagram
$$
\xymatrix{
\Emb_{S_r}(M)\ar@/_3ex/[dr]_(.3){\Diff(S_r;o_{S_r},\Sigma,o_\Sigma)}\ar@{=}[r]&\pFlag_{\mathcal S,\iota}(M)\ar[d]^-{\Diff(\mathcal S;\iota,o_{\mathcal S})}\ar@{^(->}[r]&\pFlag_{\mathcal S}(M)\ar[d]^{\Diff(\mathcal S;o_{\mathcal S})}
\\
&\Flag^\ori_{\mathcal S,\iota,o_{\mathcal S}}(M)\ar@{^(->}[r]&\Flag^\ori_{\mathcal S,o_{\mathcal S}}(M),
}
$$
which may be regarded as a reduction of structure groups along the inclusion $\Diff(S_r;o_{S_r},\Sigma,o_\Sigma)=\Diff(\mathcal S;\iota,o_{\mathcal S})\subseteq\Diff(\mathcal S,o_{\mathcal S})$.
As before, the arrows indicating principal bundles are labeled with their structure groups.

\subsection{Manifolds of closed currents}\label{SS:currents}

For notational simplicity we assume $\dim(M)=n$ and $\dim(S)=k$ in this paragraph.
Integration provides a natural $\Diff_c(M)$ equivariant embedding of $\Gr^\ori_S(M)$ into the currents, i.e., distributional forms on $M$:
\begin{equation}\label{E:currentGr}
\Gr^\ori_S(M)\to\Omega^k(M)'=\Gamma_c^{-\infty}(\Lambda^{n-k}T^*M\otimes\mathcal O_M),\qquad\langle(N,o),\alpha\rangle:=\int_N\alpha\otimes o,
\end{equation}
where $\alpha\in\Omega^k(M)$.
This map, which resembles the classical Pl\"ucker embedding, is readily seen to be a smooth injective immersion.
The currents in its image are all closed by Stokes' theorem.
The cohomology class represented by the current associated with $(N,o)$ corresponds to its fundamental class via Poincar\'e duality,
\[
H_c^{n-k}(M;\mathcal O_M)\cong H_k(M).
\]
Clearly, the image of the open and closed subset $\Gr_S^\ori(M)_{(N,o)}$ of $\Gr_S^\ori(M)$ under the map \eqref{E:currentGr} consists of a single $\Diff_c(M)$ orbit of closed currents with very controlled singular support and wave front set.

Suppose $\mathcal S=(S_1,\dotsc,S_r)$ and consider the $\Diff_c(M)$ equivariant map obtained by composing \eqref{E:currentGr} with the canonical map $\Flag^\ori_{\mathcal S}(M)\to\Gr^\ori_{S_i}(M)$ and summing over $i$, that is,
\begin{equation}\label{E:currentFlag}
\Flag^\ori_{\mathcal S}(M)\to\Gamma_c^{-\infty}(\Lambda^*T^*M\otimes\mathcal O_M),\qquad\langle(\mathcal N,o),\alpha\rangle:=\sum_{i=1}^r\int_{N_i}\alpha\otimes o_i,
\end{equation}
where $\alpha\in\Omega^*(M)$.
If $\dim(S_i)=k_i$ and $k_1<k_2<\cdots<k_r$, then this map is a smooth injective immersion.
Clearly, its image consists of (nonhomogeneous) closed currents in $M$.
By equivariance, the image of the open and closed subset $\Flag_{\mathcal S}^\ori(M)_{(\mathcal N,o)}$ of $\Flag_{\mathcal S}^\ori(M)$ under the map \eqref{E:currentFlag} forms a single $\Diff_c(M)$ orbit of closed (nonhomogeneous) currents.

\subsection{Transgression to nonlinear Grassmannians}\label{SS:transGr}

We first recall the natural transgression of differential forms on $M$ 
to differential forms on the nonlinear Grassmannian $\Gr_S^\ori(M)$ of oriented submanifolds \cite[Section~2]{HV}.
Each $\al\in\Om^{\dim (S)+\ell}(M)$ induces $\tilde\al\in\Om^\ell(\Gr_S^\ori(M))$ by
\begin{equation}\label{100}
(\tilde\al)_N(\xi_1,\dots,\xi_\ell):=\int_Ni_{\xi_\ell}\dots i_{\xi_1}\al,
\quad\xi_i\in\Ga(TM|_{N}/TN).
\end{equation}
Clearly, the assignment $\alpha\mapsto\tilde\alpha$ is $\Diff(M)$ equivariant.
Moreover, the following identities hold \cite[Lemma~1]{HV}:
\begin{equation}\label{calc}
 \widetilde{d\al}=d\tilde\al; \qquad
i_{\ze_X}\tilde\al=\widetilde{i_X\al};\qquad
L_{\ze_X}\tilde\al=\widetilde{L_X\al}.
\end{equation}
Here $\ze_X$ denotes the infinitesimal action of $X\in\mathfrak X(M)$ on $\Gr^\ori_S(M)$.

Let $S$ be endowed with an orientation $o_S$.
Using the  fiber integral for the trivial $S$-bundle 
$\pr_1:\Emb_S(M)\x S\to\Emb_S(M)$, we define
\[
\hat\al:=(\pr_1)_!(\ev^*\al\otimes\pr_2^*o_S)\in\Om^\ell(\Emb_S(M)),
\] 
where $\ev:\Emb_S(M)\x S\to M$
denotes the evaluation map and $\pr_2:\Emb_S(M)\x S\to S$.
This is a basic form for the principal bundle projection $\pi_{o_S}\colon\Emb_S(M)\to\Gr^\ori_{S,o_S}(M)$, $\pi_{o_S}(\varphi)=(\varphi(S),\varphi_*o_S)$, cf.~\eqref{pior}, hence it descends to a form on $\Gr_{S,o_S}^\ori(M)$.
This is exactly the restriction of the transgression $\tilde\al$, thus $\hat\al=\pi_{o_S}^*\tilde\al$ (see \cite{Vizman}).

A more elegant way to obtain the transgressed form $\tilde\al$ uses the tautological bundle.
Let $\T^\ori$ denote the pullback of the tautological bundle $p:\T\to\Gr_S(M)$
by the forgetful map $\Gr^\ori_S(M)\to\Gr_S(M)$.
More concretely, we get the $S$-bundle
\[
p^\ori:\T^\ori=\{(N,o,x)\in\Gr_S^\ori(M)\times M:x\in N\}\to\Gr^\ori_S(M).
\]
Let $q^\ori:\T^\ori\to M$ denote the projection on the last factor. 
Then the transgression of $\al$ to $\Gr^\ori_S(M)$ can be expressed in the form
\begin{equation}\label{E:talpha}
\tilde\al=(p^\ori)_!((q^\ori)^*\al\otimes o_{p^\ori}),
\end{equation}
where $o_{p^\ori}\in\Ga(\mathcal O_{\ker Tp^\ori})$ denotes the canonical orientation of the vertical bundle of $p^\ori$ induced by identification of the fiber over $(N,o)$ with $N$ via the restriction of $q^\ori$.
Indeed, defining $\tilde\pi_{o_S}(\ph,x)=(\ph(S),\ph_*o_S,\ph(x))$, we obtain a commutative diagram 
$$
\xymatrix{
&M\\
\Emb_S(M)\x S\ar@/^2ex/[ur]^-\ev\ar[d]_-{\pr_1}\ar[rr]^-{\tilde\pi_{o_S}}&&\T^\ori\ar[d]^-{p^\ori}\ar@/_2ex/[ul]_-{q^\ori}
\\
\Emb_{S}(M)\ar[rr]^-{\pi_{o_S}}&&\Gr_S^\ori(M),
}
$$
where the rectangle is a pullback diagram.
Moreover, $\tilde\pi_{o_S}^*o_{p^\ori}=\pr_2^*o_S$.
Using the fact that integration along the fiber commutes with pullbacks \cite[7.12]{GHV} one obtains
\[
\pi_{o_S}^*\bigl((p^\ori)_!((q^\ori)^*\al\otimes o_{p^\ori})\bigr)
=(\pr_1)_!\tilde\pi_{o_S}^*\bigl((q^\ori)^*\al\otimes o_{p^\ori}\bigr)
=(\pr_1)_!(\ev^*\al\otimes\pr_2^*o_S)
=\hat\alpha
=\pi_{o_S}^*\tilde\alpha,
\]
and thus \eqref{E:talpha}, because the map $\pi_{o_S}$ is a submersion that covers any given connected component of $\Gr_S^\ori(M)$, for a suitable choice of $o_S$.

\subsection{Transgression to manifolds of nonlinear flags}\label{SS:transFlag}

It works similarly for the transgression of differential forms to the manifold $\Flag_{\mathcal S}^\ori(M)$ of oriented nonlinear flags.
We start with a collection of differential forms on $M$:
\[
\al=(\al_i),\quad\al_i\in\Om^{\dim (S_i)+\ell}(M).
\]
The transgression to $\Flag^\ori_{\mathcal S}(M)$ can be defined with the help of the transgression \eqref{100} to nonlinear Grassmannians by
\begin{equation}\label{delfin}
\tilde\al:=\sum_{i=1}^r\pr_i^*\tilde{\al_i}\in\Om^\ell(\Flag^\ori_{\mathcal S}(M)),
\end{equation}
with $\pr_i:\Flag^\ori_{\mathcal S}(M)\to\Gr^\ori_{S_i}(M)$ the projection on the $i$-th factor.

As above, there are two further descriptions of $\tilde\al$ via fiber integration.
The first one uses the pullback $\mathcal T_i^\ori$ of the tautological bundle $\mathcal T_i$ over $\Flag_{\mathcal S}(M)$ from Proposition~\ref{P:taut}, namely
\[
\mathcal T_i^\ori=\{((N_1,o_1),\dotsc,(N_r,o_r);x)\in\Flag^\ori_{\mathcal S}(M)\times M:x\in N_i\}.
\]
Let $p_i^\ori:\T_i^\ori\to \Flag_{\mathcal S}^\ori(M)$ denote the bundle projections and $q_i^\ori:\T^\ori_i\to M$.
Then the transgression $\tilde\alpha$ can be expressed using fiber integration along $\T_i^\ori$ in the form
\begin{equation}\label{E:tali}
\tilde\al=\sum_{i=1}^r(p_i^\ori)_!((q_i^\ori)^*\al_i\otimes o_{p_i^\ori}),
\end{equation}
where $o_{p_i^\ori}\in\Ga(\mathcal O_{\ker Tp_i^\ori})$ denotes the canonical orientation of the vertical bundle of $p_i^\ori$.
This follows from the right hand side of the subsequent commutative diagram, where the lower right rectangle is a pullback, cf.\ Proposition~\ref{P:taut}(\ref{I:taut:TiT}), using \eqref{E:talpha} and the fact that integration along the fiber commutes with pullbacks:
$$
\xymatrix{
&&M\\
\pFlag_{\mathcal S}(M)\x S_i\ar@/^2ex/[urr]^-{\ev_i}\ar[d]_-{\pr_1}\ar[rr]^-{\tilde\pi_{o_{\mathcal S}}}&&\T^\ori_i\ar[d]^-{p^\ori_i}\ar[u]_-{q^\ori_i}\ar[rr]^-{\tilde\pr_i}&&\mathcal T^\ori\ar[d]^-{p^\ori}\ar@/_2ex/[ull]_-{q^\ori}
\\
\pFlag_{\mathcal S}(M)\ar[rr]^-{\pi_{o_{\mathcal S}}}&&\Flag_{\mathcal S}^\ori(M)\ar[rr]^-{\pr_i}&&\Gr_{S_i}^\ori(M).
}
$$
For the other description we choose orientations $o_{\mathcal S}=(o_{S_1},\dotsc,o_{S_r})$ as in Section~\ref{SS:oFlag}. 
These give rise to the left hand side of the commutative diagram, where the lower left rectangle is a pullback too.
Using \eqref{E:tali} and proceeding as above, we obtain
$$
\pi_{o_{\mathcal S}}^*\tilde\alpha=\sum_{i=1}^r(\pr_1)_!\bigl(\ev_i^*\alpha_i\otimes\pr_2^*o_{S_i}\bigr).
$$
This completely characterizes $\tilde\alpha$ since the maps $\pi_{o_{\mathcal S}}$ are submersions covering all connected components of $\Flag^\ori_{\mathcal S}(M)$, as $o_{\mathcal S}$ varies over all possible orientations.

\section{Coadjoint orbits of symplectic nonlinear flags}\label{S:Flagsymp}

As an application of the results presented above, we will now discuss how certain coadjoint orbits of the Hamiltonian group $\Ham(M)$ of a closed symplectic manifold can be parametrized by nonlinear flag manifolds, cf.\ Theorem~\ref{teo2} below. 
This generalizes \cite[Theorem~3]{HV} about symplectic nonlinear Grassmannians (recalled in the first subsection below). 
We consider the manifold $\Flag_{\mathcal S}^\symp(M)$ of symplectic nonlinear flags, an open subset of $\Flag_{\mathcal S}(M)$. Using a transgression procedure (similar to the one for oriented nonlinear flags) we endow it with a natural symplectic form.
We show that the momentum map for the $\Ham(M)$ action realizes connected components of the symplectic manifold $\Flag_{\mathcal S}^\symp(M)$ as coadjoint orbits of $\Ham(M)$.

\subsection{Symplectic nonlinear Grassmannians}

Let $M$ be a closed manifold endowed with a symplectic form $\om$, and let $S$ be a closed $2k$-dimensional manifold.
The \emph{symplectic nonlinear Grassmannian} $\Gr_{S}^{\symp}(M)$ of symplectic submanifolds of $(M,\om)$ of type $S$, introduced and studied in \cite{HV}, is an open subset of the nonlinear Grassmannian $\Gr_S(M)$. 
Restricting the fundamental frame bundle in \eqref{E:principal} to $\Gr_S^\symp(M)$, we obtain a smooth principal bundle
\begin{equation}\label{sympuls}
\pi\colon\Emb_S^{\symp}(M)\to\Gr_S^{\symp}(M),
\end{equation} 
with the same structure group, $\Diff(S)$, where 
\[
\Emb^{\symp}_{S}(M)=\{\ph\in\Emb_S(M)\mid\ph^*\om\in\Om^2(S)\text{ symplectic}\}
\]
denotes the open subset of symplectic embeddings in $\Emb_S(M)$.
The group $\Symp(M)$ of symplectic diffeomorphisms acts on the manifold of symplectic embeddings into $M$, as well as on the symplectic nonlinear Grassmannian of $M$, and the principal bundle \eqref{sympuls} is $\Symp(M)$ equivariant.

A transgression functor similar to the one considered in Section~\ref{SS:transGr} exists for the symplectic nonlinear Grassmannian:
\[
\Om^{2k+\ell}(M)\ni\alpha\mapsto\tilde\al\in\Om^\ell(\Gr_S^\symp(M)).
\]
It has the same expression as in \eqref{100}, but no orientation is needed now, since the symplectic submanifolds are naturally oriented by their induced Liouville volume forms.
It also has similar functorial properties to the tilde calculus on oriented nonlinear Grassmannians \eqref{calc}.

Again there is a way to obtain the transgressed form $\tilde\al$ with a tautological bundle.
Let 
\[
p\colon\T^\symp\to\Gr_S^\symp(M)
\] 
denote the restriction of the tautological bundle $\T$ to the open subset $\Gr^\symp_S(M)\subseteq\Gr_S(M)$, and let $q:\T^\symp\to M$ denote the projection on the $M$ factor. 
Then
\[
\tilde\al=p_!(q^*\al\otimes o^\omega_p),
\]
where $o_p^\omega\in\Ga(\mathcal O_{\ker Tp})$ is the canonical orientation of the vertical bundle of $\T^\symp$ that comes from the orientation by the Liouville volume form of the fiber over the symplectic submanifold $N$, fiber identified to $N$ via $q$.

For the rest of this paragraph we follow \cite{HV}.
The symplectic nonlinear Grassmannian $\Gr_S^\symp(M)$ can be endowed with a natural symplectic form $\Om=\tfrac{1}{k+1}\widetilde{\om^{k+1}}$.
More precisely,
\begin{equation}\label{alfa}
\Om_N(\xi,\et):=\frac{1}{k+1}\int_Ni_\et i_\xi\om^{k+1},\quad\xi,\et\in T_N\Gr_S^{\symp}(M)=\Ga(TM|_N/TN),
\end{equation}
where the orientation on the $2k$-dimensional symplectic submanifold $N$ is the one induced by the Liouville volume form.

The Lie algebra of the Hamiltonian group $\Ham(M)$ is $\hamoe(M)$, the Lie algebra of Hamiltonian vector fields.
Since $M$ is compact, $\hamoe(M)$ can be identified with the Lie algebra $C_0^\infty(M)$ of functions with zero integral on each connected component, endowed with the Poisson bracket.
The action of $\Ham(M)$ on $\Gr_{S}^{\symp}(M)$ is transitive on connected components \cite[Proposition~3]{HV}.
Moreover, the action is Hamiltonian with injective $\Symp(M)$ equivariant moment map
\begin{equation}\label{ochi}
J:\Gr_{S}^{\symp}(M)\to C_0^\oo(M)^*=\hamoe(M)^*,\quad J(N)(f)=\int_N f\om^k.
\end{equation}
Indeed, functorial identities analogous to \eqref{calc} ensure that $i_{\ze_{X_f}}\Om=d(\widetilde{f\om^k})$, where the function $\widetilde{f\om^k}$ maps $N$ to $\int_Nf\om^k$.
The next result follows now by using a well known fact, also recalled in Proposition~\ref{spcase}.

\begin{theorem}[{\cite[Theorem 3]{HV}}]\label{old}
The restriction of the moment map $J\colon\Gr_S^\symp(M)\to\hamoe(M)^*$ in \eqref{ochi} to any connected component of $\Gr_S^\symp(M)$ is one-to-one onto a coadjoint orbit of the Hamiltonian group $\Ham(M)$. 
The Kostant--Kirillov--Souriau symplectic form $\om_\KKS$ on the coadjoint orbit satisfies $J^*\om_\KKS=\Om$.
\end{theorem}

Let us remark that the setting in \cite{HV} is slightly different:
there we consider the covering of $\Gr_S^\symp(M)$ that consists of oriented symplectic submanifolds of type $S$, an open subset of the oriented Grassmannian $\Gr_S^\ori(M)$, endowed with the symplectic form induced from $\frac{1}{k+1}\om^{k+1}$ by the transgression discussed in Section~\ref{SS:transGr}.

\subsection{Symplectic nonlinear flag manifolds}

We fix a sequence $\mathcal S=(S_1,\dotsc,S_r)$ of even dimensional manifolds: $\dim (S_i)=2k_i$ with 
\[
k_1<k_2<\dotsb<k_r.
\] 
The manifold of \emph{symplectic nonlinear flags of type $\mathcal S$},
$$
\Flag^\symp_{\mathcal S}(M):=\Flag_{\mathcal S}(M)\cap\prod_{i=1}^r\Gr_{S_i}^\symp(M),
$$
is an open subset of $\Flag_{\mathcal S}(M)$.
Restricting the fundamental frame bundle in \eqref{frames} to $\Flag^\symp_{\mathcal S}(M)$, we obtain a smooth principal bundle
\[
\pi\colon\pFlag^\symp_{\mathcal S}(M)\to\Flag^\symp_{\mathcal S}(M),
\]
with the same structure group, $\Diff(\mathcal S)=\prod_{i=1}^r\Diff(S_i)$, where
\[
\pFlag^{\symp}_{\mathcal S}(M)=\left\{(\varphi_1,\dotsc,\varphi_r)\in\pFlag_{\mathcal S}(M)\middle|\forall i:\ph_i^*\om\text{ symplectic}\right\}=\pFlag_{\mathcal S}(M)\cap\prod_{i=1}^{r}\Emb_{S_i}^{\symp}(M).
\]
denotes the open subset of all \emph{symplectic nonlinear frames of type $\mathcal S$.}

For each $1\le i\le r$, the differential form
\[
\al_i=\tfrac{1}{k_i+1}\om^{k_i+1}\in\Om^{2k_i+2}(M)
\]
induces a symplectic form $\Om_i:=\tilde\al_i$ on $\Gr^{\symp}_{S_i}(M)$,
by the transgression introduced in \eqref{100}.
Thus the collection $(\al_i)$ canonically induces a symplectic form on the product $\prod_{i=1}^r\Gr_{S_i}^{\symp}(M)$:
\[
\Om=\sum_{i=1}^r\pr_i^*\Om_i,
\]
where $\pr_i$ denotes the projection on the $i$-th factor.
The restriction of the symplectic form $\Om$ to the submanifold $\Flag_{\mathcal S}^{\symp}(M)\subseteq \prod_{i=1}^r\Gr_{S_i}^{\symp}(M)$, denoted again by $\Om$, can also be written as
\begin{equation}\label{omxiet}
\Om_{\N}(\xi,\et)=\sum_{i=1}^r\int_{N_i}i_{\et_i}i_{\xi_i}\al_i,\quad 
\xi=(\xi_i),\et=(\et_i)\in T_{\N}\Flag^{\symp}_{\mathcal S}(M),
\end{equation}
where each symplectic submanifold $N_i\subseteq M$ is endowed with the orientation given by the Liouville volume form.
Being an open subset of $\Flag_{\mathcal S}(M)$, the tangent space to $\Flag^\symp_{\mathcal S}(M)$ is as in \eqref{tnf}:
\[
T_{\mathcal N}\Flag^\symp_{\mathcal S}(M)
=\left\{(\xi_1,\dotsc,\xi_r)\in\prod_{i=1}^r\Gamma(TM|_{N_i}/TN_i)\middle|\forall i:\xi_{i+1}|_{N_i}=\xi_i\textnormal{ mod } TN_{i+1}|_{N_i}\right\}.
\]


\begin{proposition}\label{pro2}
The differential $2$-form $\Om$ in \eqref{omxiet} on $\Flag^{\symp}_{\mathcal S}(M)$ is symplectic.
\end{proposition}

\begin{proof}
We only have to show that $\Om$ is weakly nondegenerate. 
An arbitrary tangent vector $\xi=(\xi_i)\in\ker\Om_{\N}$ satisfies  
\[
\sum_{i=1}^r\frac{1}{k_i+1}\int_{N_i}i_{\et_i}i_{\xi_i}\om^{k_i+1}=0,
\quad\forall\et=(\et_i)\in T_{\N}\Flag^{\symp}_{\mathcal S}(M).
\]
First we consider only those  tangent vectors $\et$ with $\et_i=0$ for $1\le i\le r-1$. It follows that the restriction $\et_r|_{N_{r-1}}=0$.
We obtain the identity
\[
\frac{1}{k_r+1}\int_{N_r}i_{\et_r}i_{\xi_r}\om^{k_r+1}=\int_{N_r}\om(\xi_r,\et_r)\om^{k_r}=0
\] 
for all $\et_r\in\Ga(TM|_{N_r}/TN_r)\cong\Ga(TN_r^\om)$ that satisfy $\et_r|_{N_{r-1}}=0$.
With the help of an almost complex structure on $M$ tamed by $\om$, we deduce that $\xi_r|_{N_r\setminus N_{r-1}}=0$, and by continuity $\xi_r=0$ on whole $N_r$.
By repeating this procedure, we successively obtain that all components $\xi_i$ of $\xi$ must vanish, hence $\Om$ is nondegenerate.
\end{proof}

The action of $\Ham(M)$ on the product of symplectic manifolds $\prod_{i=1}^r\Gr_{S_i}^{\symp}(M)$ is Hamiltonian with $\Symp(M)$ equivariant moment map
\[
\bar J\colon\prod_{i=1}^r\Gr_{S_i}^{\symp}(M)\to\hamoe(M)^*,
\quad\bar J=\sum_{i=1}^r\pr_i^*J_i,
\]
where $J_i\colon\Gr_{S_i}^\symp(M)\to\hamoe(M)^*$ is the moment map \eqref{ochi} for $S=S_i$.
The manifold $\Flag_{\mathcal S}^\symp(M)$ is invariant under the action of $\Symp(M)$.
Hence, the action of $\Ham(M)$ restricted to $\Flag^\symp_{\mathcal S}(M)$ is Hamiltonian, with $\Symp(M)$ equivariant moment map $J$ given by the restriction of the moment map $\bar J$, thus
\begin{equation}\label{ibis}
J\colon\Flag^{\symp}_{\mathcal S}(M)\to C^\oo_0(M)^*=\hamoe(M)^*,
\quad\langle{J}(N_1,\dotsc,N_r),f\rangle:=\sum_{i=1}^r\int_{N_i}f\omega^{k_i},
\end{equation}
with each $N_i$ oriented by its Liouville volume form.

\begin{lemma}\label{bach}
The moment map $J$ in \eqref{ibis} is injective.
\end{lemma}

\begin{proof}
Let $\N',\N''\in\Flag_{\mathcal S}^{\symp}(M)$ such that $J(\N')=J(\N'')$.
Assume by contradiction that $N''_r\not\subseteq N_r'$.
We choose $x\in N'_r\setminus N''_r$ and a positive function $f\in C_0^\oo(M)$ with support in a small neighborhood of $x$ disjoint from $N''_r$ (hence disjoint from all $N''_i$).
We get a contradiction because $0<\sum_{i=1}^r\int_{N'_i}f\om^{k_i}=\sum_{i=1}^r\int_{N''_i}f\om^{k_i}=0$.
It follows that $N'_r=N''_r$.
We proceed in the same manner successively with all the other nested submanifolds, finally obtaining $\N'=\N''$, hence the injectivity of $J$.
\end{proof}

\subsection{Coadjoint orbits of the Hamiltonian group}

\begin{proposition}\label{unknown}
The group $\Ham(M)$ acts infinitesimally and locally transitive 
on the manifold of symplectic nonlinear flags $\Flag_{\mathcal S}^{\symp}(M)$.
\end{proposition}

\begin{proof}
For the infinitesimal transitivity let $(\xi_1,\dots,\xi_r)\in T_{\N}\Flag^{\symp}_{\mathcal S}(M)$, i.e., the normal sections $\xi_i\in\Gamma(TM|_{N_i}/TN_i)$ satisfy 
\begin{equation}\label{xix}
\xi_{i+1}|_{N_i}=\xi_i\textnormal{ mod } TN_{i+1}|_{N_i}, \quad i=1,\dots,r-1.
\end{equation}
The action of $\hamoe(M)$ on the nonlinear symplectic Grassmannian $\Gr_{S}^{\symp}(M)$ is infinitesimally transitive \cite[Proposition~3]{HV}, so there exists $h_r\in C^\oo(M)$ with
$$
\xi_r=X_{h_r}|_{N_r}\textnormal{ mod } TN_r.
$$ 
Now from \eqref{xix} follows that $X_{h_r}|_{N_{r-1}}\textnormal{ mod } TN_r|_{N_{r-1}}=\xi_{r-1}\textnormal{ mod } TN_r|_{N_{r-1}}$, so that
$$
\xi_{r-1}-X_{h_r}|_{N_{r-1}}\textnormal{ mod } TN_{r-1}\in\Ga(TN_r|_{N_{r-1}}/TN_{r-1}).
$$
Applying this time the infinitesimal transitivity of the $\hamoe(N_r)$ action on  $\Gr^{\symp}_{S_{r-1}}(N_r)$ at $N_{r-1}$,
we find $f_{r}\in C^\oo(N_{r})$ such that 
$$\xi_{r-1}-X_{h_r}|_{N_{r-1}}\textnormal{ mod } TN_{r-1}=X_{f_{r}}|_{N_{r-1}}\textnormal{ mod } TN_{r-1}. 
$$
One can always choose an extension  $\tilde f_{r}\in C^\oo(M)$ of $f_{r}$
such that its normal derivatives along the symplectic submanifold $N_r\subseteq M$ vanish (recall the decomposition 
$TM|_{N_r}=TN_r\oplus TN_r^\om$ with $TN_r^\om$ denoting the symplectic orthogonal of $TN_r$).
Hence the Hamiltonian vector field $X_{\tilde f_{r}}$ 
restricted to $N_{r}$ is equal to $X_{f_r}\in\hamoe(N_r)$.

Now define the function $h_{r-1}:=h_r+\tilde f_{r}\in C^\oo(M)$. 
Its Hamiltonian vector field  $X_{h_{r-1}}$ satisfies 
$$
\xi_r=X_{h_{r-1}}|_{N_r}\textnormal{ mod } TN_r\text{ and }\xi_{r-1}=X_{h_{r-1}}|_{N_{r-1}}\textnormal{ mod } TN_{r-1},
$$
because $X_{\tilde f_r}|_{N_r}$ is tangent to $N_r$ and $X_{\tilde f_r}|_{N_{r-1}}=X_{f_r}|_{N_{r-1}}$.

We proceed in the same manner with $\xi_{r-2},\dots,\xi_1$, obtaining in the end $h:=h_1\in C^\oo(M)$
with the properties $\xi_i=X_{h}|_{N_i}\textnormal{ mod } TN_{i}$ for all $i$.
Thus the infinitesimal generator of the Hamiltonian vector field $X_{h}$ at 
$\N\in\Flag_{\mathcal S}^{\symp}(M)$ is the tangent vector $(\xi_1,\dots,\xi_r)$
we started with.

To show local transitivity, suppose $t\mapsto\mathcal N(t)=(N_1(t),\dotsc,N_r(t))$ is a smooth curve in $\Flag^\symp_{\mathcal S}(M)$.
By infinitesimal transitivity, there exists a time dependent Hamiltonian vector field $X_t$ on $M$ such that
\begin{equation*}\label{E:Xt}
\tfrac\partial{\partial t}N_i(t)=X_t|_{N_i(t)}\textnormal{ mod } TN_i(t),\qquad1\leq i\leq r.
\end{equation*}
It is clear from the construction above that $X_t$ may be chosen to depend smoothly on $t$.
Moreover, there exist parametrizations $\varphi_i(t)\in\Emb_{S_i}(M)$ such that
\begin{equation}\label{E:phiit}
\varphi_i(t)(S_i)=N_i(t)\qquad\text{and}\qquad
\tfrac\partial{\partial t}\varphi_i(t)=X_t\circ\varphi_i(t),\qquad1\leq
i\leq r.
\end{equation}
Integrating the time dependent Hamiltonian vector field $X_t$, we obtain a smooth curve of Hamiltonian diffeomorphisms $f_t$ on $M$ such that $\frac\partial{\partial t}f_t=X_t\circ f_t$.
Combining this with \eqref{E:phiit}, we obtain $f_t\circ\varphi_i(0)=\varphi_i(t)$ and then $f_t(N_i(0))=N_i(t)$.
Hence, $f_t(\mathcal N(0))=\mathcal N(t)$, for all $t$.
Since the Hamiltonian group is locally connected by smooth arcs \cite[43.13]{KM}, we conclude that the action is locally transitive.
\end{proof}

Now a result similar to Theorem~\ref{old} follows, by using Lemma \ref{bach} and Proposition~\ref{unknown} together with the well known fact recalled in Proposition~\ref{spcase}.

\begin{theorem}\label{teo2}
The restriction of the moment map $J\colon\Flag_{\mathcal S}^\symp(M)\to\hamoe(M)^*$ in \eqref{ibis} to any connected component is one-to-one onto a coadjoint orbit of the Hamiltonian group $\Ham(M)$. 
The Kostant--Kirillov--Souriau symplectic form $\om_\KKS$ on the coadjoint orbit satisfies $J^*\om_\KKS=\Om$.
\end{theorem}

The coadjoint orbits of symplectic submanifolds can be obtained via symplectic reduction in a dual pair \cite{GBV17}.
The coadjoint orbits of symplectic nonlinear flags can also be related to a dual pair \cite{HV20}: they are obtained again by performing symplectic reduction on one leg of the dual pair.

\appendix
\section{Equivariant moment maps}

For the reader's convenience we prove here a result that belongs to mathematical folklore. The case of a nonequivariant moment map is presented in \cite[Proposition~1]{HV}.

Let $G$ be a Lie group acting from the left in a Hamiltonian way on the symplectic manifold $(\mathcal M,\Om)$, and let $\ze_X\in\mathfrak X(\mathcal M)$ denote the infinitesimal action of $X\in\g$. 
The defining identity of the moment map $J:\mathcal M\to\g^*$ is
\[
d\langle J,X\rangle=i_{\ze_X}\Om, \quad X\in\g.
\]
If $J$ is $G$ equivariant, i.e., 
\[
J(g\cdot x)=\Ad_{g^{-1}}^*J(x),\quad g\in G,\ x\in\mathcal M,
\]
then it is infinitesimally equivariant, i.e.,
\begin{equation}\label{stea}
dJ(\ze_X(x))=-\ad^*_XJ(x),\quad X\in\g,\ x\in\mathcal M.
\end{equation}
For finite dimensional $\mathcal M$, the identity \eqref{stea} is equivalent to the fact that $J$, viewed as a map $\g\to C^\oo(\mathcal M)$, is a Lie algebra homomorphism for the Poisson bracket on $C^\oo(\mathcal M)$.

\begin{proposition}\label{spcase}
Suppose the action of $G$ on $(\mathcal M,\Om)$ is transitive and infinitesimally transitive, with injective equivariant moment map $J:\mathcal M\to\g^*$.
Then $J$ is one-to-one onto a coadjoint orbit of $G$.
Moreover, it pulls back the Kostant--Kirillov--Souriau symplectic form $\omega_\KKS$ on the coadjoint orbit to $\Om$.
\end{proposition}

\begin{proof}
The first part is clear.
We denote the infinitesimal generators by $\ze_X^{\mathcal M}$ and $\ze_X^{\g^*}$.
The computation
\begin{multline*}
(J^*\om_{\KKS})_x(\ze_X^{\mathcal M}(x),\ze_Y^{\mathcal M}(x))
=(\om_{\KKS})_{J(x)}\bigl(\ze_X^{\g^*}(J(x)),\ze_Y^{\g^*}(J(x))\bigr)
=\langle J(x),[X,Y]\rangle
\\=-\langle\ad^*_YJ(x),X\rangle\stackrel{\eqref{stea}}{=}\langle dJ(\ze_Y^{\mathcal M}(x)),X\rangle=\Om_x(\ze_X^{\mathcal M}(x),\ze_Y^{\mathcal M}(x))
\end{multline*}
implies the second part.
\end{proof}

\end{document}